\documentclass[11pt,twoside,a4paper]{article}

\usepackage{amsmath}
\numberwithin{equation}{section}

\usepackage{amssymb}

\usepackage{microtype}

\usepackage[backend=biber,maxnames=10,backref=true,hyperref=true,giveninits=true,safeinputenc]{biblatex}
\ifdefined\overleaf
\else
\ifdefined\extract
\else
\fi
\fi

\DefineBibliographyStrings{english}{%
	backrefpage = {cited on page},
	backrefpages = {cited on pages},
}

\DeclareFieldFormat[report]{title}{``#1''}
\DeclareFieldFormat[book]{title}{``#1''}
\AtEveryBibitem{\clearfield{url}}
\AtEveryBibitem{\clearfield{note}}

\usepackage{graphicx}
\graphicspath{{images/}}

\usepackage{tikz}
\usepackage{pgfplots}
\usetikzlibrary{external,intersections,calc,decorations.pathreplacing}
\tikzsetexternalprefix{tikz/}
\usepackage[all]{xy}

\newcommand{\ourtitle}{Explicit Inversion of the Attenuated Photoacoustic Operator in General Observation Geometries}

\title{\ourtitle}
\author{Cong Shi$^{1, 2}$\\{\footnotesize\href{mailto:cong.shi@univie.ac.at}{cong.shi@univie.ac.at}}}
\usepackage[pdftex,colorlinks=true,linkcolor=blue,citecolor=green,urlcolor=blue,bookmarks=true,bookmarksnumbered=true]{hyperref}
\makeatletter
\AtBeginDocument{
	\hypersetup
	{
	    pdfauthor={Cong Shi},
	    pdfsubject={Attenuation in Photoacoustic Imaging},
	    pdftitle={\ourtitle},
	    pdfkeywords={Photoacoustic Imaging, Attenuation, Reconstruction formula, Planar measurement geometry, Sphere measurement geometry}
	}
}
\makeatother

\makeatletter
\let\oldtheequation\theequation
\renewcommand\tagform@[1]{\maketag@@@{\ignorespaces#1\unskip\@@italiccorr}} 
\renewcommand\theequation{(\oldtheequation)}
\makeatother
\def\equationautorefname~{}

\usepackage[hyperref,amsmath,thmmarks]{ntheorem}
\usepackage{aliascnt}

\newtheorem{lemma}{Lemma}[section]

\newaliascnt{proposition}{lemma}
\newtheorem{proposition}[proposition]{Proposition}
\aliascntresetthe{proposition}

\newaliascnt{corollary}{lemma}
\newtheorem{corollary}[corollary]{Corollary}
\aliascntresetthe{corollary}

\newaliascnt{theorem}{lemma}
\newtheorem{theorem}[theorem]{Theorem}
\aliascntresetthe{theorem}

\theorembodyfont{\normalfont}
\newaliascnt{definition}{lemma}
\newtheorem{definition}[definition]{Definition}
\aliascntresetthe{definition}

\newaliascnt{assumption}{lemma}

\aliascntresetthe{assumption}

\theoremstyle{nonumberplain}
\theoremseparator{:}
\theoremheaderfont{\normalfont\itshape}

\newtheorem{remark}{Remark}

\theoremsymbol{\ensuremath{\square}}
\newtheorem{proof}{Proof}

\usepackage[a4paper,centering,bindingoffset=0cm,marginpar=2cm,margin=2.5cm]{geometry}

\usepackage[pagestyles]{titlesec}
\titleformat{\section}[block]{\large\sc\filcenter}{\thesection.}{0.5ex}{}[]
\titleformat{\subsection}[runin]{\bf}{\thesubsection.}{0.5ex}{}[.]

\makeatletter
\AtBeginDocument{
\newpagestyle{headers}%
{%
	\headrule
	\sethead%
		[\footnotesize\thepage]%
		[\footnotesize\sc Cong Shi]%
		[]%
		{}%
		{\footnotesize\sc\ourtitle}%
		{\footnotesize\thepage}
	\setfoot{}{}{}
}
\pagestyle{headers}}
\makeatother

\usepackage[font=footnotesize,format=plain,labelfont=sc,textfont=sl,width=0.75\textwidth,labelsep=period]{caption}

\usepackage{dsfont}
\newcommand{\N}{\mathds{N}}

\newcommand{\R}{\mathds{R}}
\newcommand{\C}{\mathds{C}}
\renewcommand{\H}{\mathds{H}}

\let\RE\Re
\let\Re=\undefined
\DeclareMathOperator{\Re}{\RE e}
\let\IM\Im
\let\Im=\undefined
\DeclareMathOperator{\Im}{\IM m}
\DeclareMathOperator{\supp}{supp}
\DeclareMathOperator{\sgn}{sgn}

\newcommand{\vx}{\mathbf{x}}

\newcommand{\vy}{y}

\newcommand{\vxi}{\boldsymbol{\xi}}

\newcommand{\e}{\mathrm e}
\let\ii\i
\renewcommand{\i}{\mathrm i}
\renewcommand{\d}{\,\mathrm d}

\newcommand{\Ff}{\mathcal{F}}

\renewcommand{\S}{\mathcal{S}}

\usepackage{comment}

\begin{document}

\maketitle
\hspace*{1em}
\parbox[t]{0.49\textwidth}{\footnotesize
\hspace*{-1ex}$^1$Computational Science Center\\
University of Vienna\\
Oskar-Morgenstern-Platz 1\\
A-1090 Vienna, Austria}
\parbox[t]{0.4\textwidth}{\footnotesize
\hspace*{-1ex}$^2$Johann Radon Institute for Computational\\
\hspace*{1em}and Applied Mathematics (RICAM)\\
Altenbergerstra{\ss}e 69\\
A-4040 Linz, Austria}
\vspace*{2em}

\begin{abstract}
In this paper, we derive explicit reconstruction formulas for two common measurement geometries: a plane and a sphere. The problem is formulated as inverting the forward operator $R^a$, which maps the initial source to the measured wave data. Our first result pertains to planar observation surfaces. By extending the domain of $R^a$ to tempered distributions, we provide a complete characterization of its range and establish that the inverse operator $(R^a)^{-1}$ is uniquely defined and "almost" continuous in the distributional topology. Our second result addresses the case of a spherical observation geometry. Here, with the operator acting on $L^2$ spaces, we derive a stable reconstruction formula of the filtered backprojection type.
\end{abstract}

\section{Introduction}

In photoacoustic tomography (PAT), an initial pressure distribution --- induced by the absorption of a short electromagnetic pulse in a biological medium --- is recovered from acoustic measurements taken at its boundary. A significant challenge arises from acoustic attenuation, where viscous effects in biological tissues cause frequency-dependent energy loss, complicating accurate reconstruction. While exact inversion formulas are well-established for non-attenuating media, none are available for the key geometries of a plane or a sphere in attenuating media. Current approaches for these cases, such as time-reversal \cite{AmmBreGarWah11} or corrected filtered backprojection methods \cite{SchShi18}, offer only approximate solutions. This paper addresses this gap by deriving the first exact, analytical reconstruction formulas in the frequency domain for these two geometries, thereby providing a fundamental solution to the attenuated PAT problem.

Photoacoustic tomography (PAT) in non-attenuating media is a well-established field, where the forward model reduces to the classical spherical mean operator $R$. For planar geometries, a rich body of work provides exact reconstruction formulas, broadly categorized into temporal back-projection \cite{And88, Faw85, BurBauGruHalPal07, XuWan05} and highly efficient Fourier-domain methods \cite{Faw85, And88, HalSchZan09b, XuFenWan02, XuXuWan02, AnaZhaModRiv07}. The latter utilize the Fast Fourier Transform (FFT) for computational efficiency. Similarly, for closed surfaces such as a sphere, filtered backprojection (FBP) algorithms offer a powerful framework that combines theoretical elegance with practical utility \cite{Rad17b}.

Current reconstruction methods in PAT are increasingly leaning towards data-driven and computationally intensive model-based approaches, in order to overcome ill-posedness and model imperfections. Deep learning continues to be a dominant force in this area, with innovations like the pattern framework demonstrating powerful cross-modal analysis for complex structures like the whole brain \cite{CheYanLuoNiuYu24}. Meanwhile, significant advancements are also being made in model-based techniques. A main research direction is given by simultaneous reconstruction of multiple biological and acoustic parameters, such as initial pressure and spatially varying speed of sound, to correct for tissue heterogeneity and improve quantitative accuracy, as seen in the work of Suhonen et al. \cite{SuhLucPulArrCox25_report} and \cite{KiaUhl25}. Other innovations focus on optimizing the imaging process itself, from enhancing image quality via novel sampling strategies \cite{HakRajKavPra25} to determining the optimal illumination function \cite{HuyKal25}. While these methods show impressive empirical performance, a key criticism of the deep learning approaches is their lack of interpretability; they often function as "black boxes" without providing insight into the underlying physics. Similarly, advanced joint-reconstruction methods, though physics-based, are often computationally prohibitive, iterative, and lack closed-form solutions (see, e.g., \cite{KulRamSahTar25_report}). This situation calls for explicit, model-based reconstruction formulas that are both analytically rigorous and computationally efficient, particularly for attenuating media where researches are still in an early phase.

However, these advancements share a critical limitation: they rely on the assumption of a lossless medium. In biological tissues, acoustic attenuation -- the frequency-dependent dissipation of wave energy -- fundamentally alters the physics of wave propagation. This deviation from the standard model introduces significant blurring and artifacts, making classical inversion formulas inadequate. While current state-of-the-art approaches for attenuating media, such as time-reversal \cite{AmmBreGarWah11} or approximate FBP corrections \cite{SchShi18}, offer pragmatic solutions, they are inherently iterative or approximate. A fundamental question remains unanswered: does an exact, analytical inversion formula exist for the attenuated operator $R^a$?

This paper answers this question definitively by deriving the first exact reconstruction formulas for the attenuated PAT problem in two canonical geometries: a plane and a sphere. Our work moves beyond numerical compensation and establishes a new analytical foundation for quantitative imaging in attenuating media.

Our approach is twofold:

\textbf{For Planar Geometry:} We introduce a novel factorization of the attenuated operator $R^a = A \circ R$ within the space of tempered distributions. We then systematically construct the inverses of both components: the classical $R$ and the attenuating operator $A$. This composition of inverses yields a new, exact closed-form solution, for which we prove uniqueness and stability. The conceptual foundation for this distributional framework builds upon the pioneering work of \cite{And88}.

\textbf{For Spherical Geometry:} We reformulate the inverse problem by constructing a compact operator $T$ that precisely encodes the attenuation. This allows us to express the inverse of $R^a$ as a Neumann series expansion, $(I+T)^{-1}$, applied to a modified FBP operator. This result elegantly generalizes the classical FBP formula \cite{Rad17b} to account for attenuation exactly, providing a direct and computationally structured reconstruction method.

In essence, we provide exact inversion formulas (shown in \autoref{thm:main}) to solve the inverse problem, moving beyond the approximate reconstructions prevalent in the literature. To the best of our knowledge, these constitute the first such analytical solutions for these specific geometries. They offer not only practical algorithms but also a deeper theoretical understanding of the attenuated PAT problem. Furthermore, we derive two explicit analytic reconstruction formulas: one for a planar observation surface \autoref{thm:plane} and another for a spherical surface \autoref{thm:sphere}. Finally, we prove the existence and uniqueness of the solutions achievable with these formulas in \autoref{sec:planar} and \autoref{sec:sphere}, respectively.

The paper is structured as follows: \autoref{sec:Amodel} formalizes the attenuated PAT model. \autoref{sec:2formulas} provides a concise overview of our two main explicit reconstruction formulas. \autoref{sec:planar} details the derivation, implementation, and uniqueness proof for the planar case. \autoref{sec:sphere} presents the derivation of the novel FBP-based algorithm for the spherical geometry. We conclude with a discussion on the implications of our findings.

\section{photoacoustic tomography with considering attenuation}\label{sec:Amodel}

In standard photoacoustic imaging, see e.g.~\cite{Wan09}, it is assumed that the medium is \emph{non-attenuating}, and the imaging problem consists in visualising the spatially, compactly supported \emph{absorption density function} $h:\R^3 \to \R$, appearing as a source term in the wave equation
\begin{equation} \label{eqWaveEquation}
\begin{aligned}
\partial_{tt}p(t,\vx)-\Delta p(t, \vx)&=\delta'(t)h(\vx),\quad &&t\in\R,\;\vx \in \R^3, \\
p(t,\vx)&=0,\quad&&t<0,\;\vx\in\R^3,
\end{aligned}
\end{equation}
from the measurements.

Define the $n-$dimensional Fourier transform of a function $f(\vx)$ through
\begin{equation}\label{FourierTransform}
(\Ff f)(\mathbf{\eta})  = \frac1{(2\pi)^{n/2}}\int_{\R^n} f(\vx)\e^{\i\eta \cdot \vx}\d \vx,
\end{equation}
and we will use $\Ff_1, \Ff_2, \dots$ to denote Fourier transforms with respect to the $1st$, $2nd$,... variables.

The inverse Fourier transform will be similarly defined as
\begin{equation}\label{InverseFourierTransform}
(\Ff^{-1} f)(\vx)  = \frac1{(2\pi)^{n/2}}\int_{\R^n} f(\mathbf{\eta})\e^{-\i\vx \cdot \mathbf{\eta}}\d \vx,
\end{equation}
and $\Ff^{-1}_1, \Ff^{-1}_2, \dots$ denote inverse Fourier transforms with respect to the  $1st$, $2nd$,... variables.

In the case where the function being transformed is supported on the positive real axis, we define the \emph{Fourier-Laplace Transform} as an extension to the ordinary Fourier transform by the same formula \autoref{FourierTransform}, but with the arguments $\mathbf{\eta}$ in the closed upper-half complex plane $\overline{\H^+}$. By abusing of notation, we also use $\Ff$ and $\Ff_1, \Ff_2, \dots$ to represent these transforms. According to the Paley-Wiener Theorem, the result of such a transform is an analytic function in $\H^+$.

In this paper, we consider photoacoustic imaging in attenuating media. if we denote $p^a$ as the notation in attenuation media for distinguishing the pressure $p$ in the no-attenuation media, then the propagation of the waves is described by the attenuated wave equation
\begin{equation}\label{eqWaveEquationAttenuation}
\begin{aligned}
\mathcal A_\kappa p^a(t,\vx)-\Delta p^a(t,\vx)&=\delta'(t)h(\vx),\quad &&t\in\R,\;\vx\in\R^3, \\
p^a(t,\vx)&=0,\quad&&t<0,\;\vx\in\R^3,
\end{aligned}
\end{equation}
where $\mathcal A_\kappa$ is the pseudo-differential operator defined in frequency domain by:
\begin{equation}\label{eqAttenuationOp}
\Ff_1 \mathcal (A_\kappa p^a)(\omega,\vx) = -\kappa^2(\omega)\Ff_1 p^a(\omega,\vx),\quad \omega \in \R,\;\vx \in \R^3,
\end{equation}
for some attenuation coefficient $\kappa:\R\to\C$ which admits a solution \autoref{eqWaveEquationAttenuation}.

The formal difference between \autoref{eqWaveEquation} and \autoref{eqWaveEquationAttenuation} is that the second time derivative operator $\partial_{tt}$ is replaced by a pseudo-differential operator $\mathcal A_\kappa$.
We emphasise that standard photoacoustic imaging corresponds to $\kappa^2(\omega) = \omega^2$.

The \emph{attenuated photoacoustic imaging} problem consists in estimating $h(\vx)$ with $\vx$ in the object $\Omega_\epsilon$ from
measurements $p^a(t, \vx)$ for $(t, \vx)\in (0, \infty)\times\partial\Omega$, where $\partial \Omega$
is the boundary of a domain $\Omega$ which contains the support of $h$.

\begin{definition}\label{deAttCoeff}
We call a non-zero function $\kappa\in C^\infty(\R;\overline\H)$, where $\H^+=\{z\in\C\mid\Im z>0\}$ denotes the upper half complex plane and $\overline{\H^+}$ its closure in $\C$, an \emph{attenuation coefficient} if
\begin{enumerate}
\item\label{enAttCoeffPolBdd}
all the derivatives of $\kappa$ are polynomially bounded. That is, for every $\ell\in\N_0$ there exist constants $\kappa_1 > 0$ and $N \in \N$ such that
\begin{equation}\label{eqAttCoeffPolBdd}
|\kappa^{(\ell)}(\omega)| \le \kappa_1(1 + |\omega|)^N,
\end{equation}
\item\label{enAttCoeffHol}
there exists a holomorphic continuation $\tilde\kappa:\overline{\H^+}\to\overline{\H^+}$ of $\kappa$ on the upper half plane, that is, $\tilde\kappa\in C(\overline{\H^+}; \overline{\H^+})$ with $\tilde\kappa|_{\R} = \kappa$ and $\tilde\kappa : \H^+ \to \overline{\H^+}$ is holomorphic; with
\[ |\tilde\kappa(z)|\le \tilde\kappa_1(1+|z|)^{\tilde N}\quad\text{for all}\quad z\in\overline{\H^+} \]
for some constants $\tilde\kappa_1>0$ and $\tilde N\in\N$.
\item\label{enAttCoeffSymm}
we have the symmetry $\kappa(-\omega) = -\overline{\kappa(\omega)}$ for all $\omega\in\R$.
\end{enumerate}
\end{definition}

The meaning of these three conditions are mentioned in \cite{ElbSchShi17}. Here we briefly introduce where we will use them. The first condition is used for proving the solution $p^a$ of attenuated wave equation \autoref{eqWaveEquationAttenuation} belongs to the Schwartz space when the observation surface is a plane. The second condition is to prove the continuity of the operator which relates $p(t, \vxi)$ and $p^a(t, \vxi)$.

\begin{definition}\label{deWeakAttenuation}
We call an attenuation coefficient $\kappa\in C^\infty(\R; \overline\H)$ a \emph{weak attenuation coefficient} if it is of the form
\[ \kappa(\omega) = \frac\omega c+\i\kappa_\infty+\kappa_*(\omega), \quad\omega\in\R, \]
for some constants $c>0$ and $\kappa_\infty\ge0$ and a bounded function $\kappa_*\in C^\infty(\R)\cap L^2(\R)$.
\end{definition}

For the plane case, the boundedness of $\Im \kappa_{*}$ is needed. For the sphere case, $\kappa_{*}\in L^2(\R)$ is used for the doing the taylor expansion.

Since we're considering the equation in the space of distributions, we will first write the attenuated wave equations in a weak form. A theorem from \cite{ElbSchShi17} guarantees the existence and uniqueness of the weak solution.

\begin{proposition}\label{thSolAttWave}
Let $\kappa$ be an attenuation coefficient and  $h\in \mathcal{S}(\R^3)$ with Schwartz space $\mathcal{S}$ defined in \autoref{SchwartzSpace}.
	
Then, the attenuated wave equation in the frequency domain
\begin{equation}\label{eqSolAttWaveEq}
\left\langle\kappa^2(\omega) \Ff_1 p^a,\vartheta\right \rangle_{\mathcal S',\mathcal S}+\left\langle \Ff_1 p^a,\Delta\vartheta\right \rangle_{\mathcal S',\mathcal S} = -\int_{\R\times\R^3}\i\omega h(\vx)\Ff_1\vartheta(\omega,\vx)\d\omega\,\d \vx,\quad\vartheta\in\mathcal S(\R\times\R^3),
\end{equation}
has a unique solution $p^a\in\mathcal S'(\R\times\R^3)$ with $\supp p^a\subset[0,\infty)\times\R^3$ and space of tempered distribution $\mathcal S'$ defined in \autoref{DistributionSpace}.

Moreover, $p^a$ is of the form
\begin{equation}\label{eqSolAttWave}
\left\langle p^a,\varphi\right \rangle_{\mathcal S',\mathcal S} =  \int_{-\infty}^\infty\int_{\R^3}\int_{\R^3}G_\kappa(\omega,\vx-\vy)h(\vy)\d \vy\,\Ff_1 \varphi(\omega; \vx)\d \vx\d\omega,
\end{equation}
where $\Ff_1 \varphi$ denotes the Fourier transform of $\varphi$ with respect to $t$ and $G$ denotes the integral kernel
\begin{equation*}\label{eqSolAttWaveKernel}
G_\kappa(\omega,\vx) = -\frac{\i\omega}{4\pi\sqrt{2\pi}}\frac{\e^{\i\kappa(\omega)|\vx|}}{|\vx|},\quad\omega\in\R,\;\vx\in\R^3\setminus\{0\}.
\end{equation*}
\end{proposition}

According to \autoref{thSolAttWave}, the solution $p^a$ of the attenuated wave equation is given by \autoref{eqSolAttWave}. This means that the temporal Fourier transform of $p^a$ is the regular distribution corresponding to the function
\begin{equation}\label{freqsol-att}
\Ff_1 p^a(\omega, \vxi) = -\int_{\R^3} \frac{\i \omega}{4\pi\sqrt{2\pi}}\frac{\e^{\i\kappa(\omega)|\vxi-\vy|}}{|\vxi-\vy|}h(\vy)\d \vy.
\end{equation}

Specifically, when there is no attenuation, which means $\kappa(\omega)=\omega$, the attenuated wave equation \autoref{eqSolAttWaveEq} reduces to the standard wave equation \autoref{eqWaveEquation}. Consequently, the solution $p$ of the standard wave equation \autoref{eqWaveEquation} is given by plugging $\kappa(\omega)=\omega$ into the formula \autoref{eqSolAttWave}. Similarly, the temporal Fourier transform of $p$ is the regular distribution corresponding to the function
\begin{equation}\label{freqsol-noatt}
\Ff_1 p(\omega, \vxi) = -\int_{\R^3} \frac{\i \omega}{4\pi\sqrt{2\pi}}\frac{\e^{\i\omega|\vxi-\vy|}}{|\vxi-\vy|}h(\vy)\d \vy.
\end{equation}

\section{Explicit reconstruction formulas}\label{sec:2formulas}

In this section, we introduce a general reconstruction formula for arbitrary observation surfaces and provide two specific examples: one for a planar surface and another for a spherical surface. The corresponding theorems are stated below, while the proofs of their validity are given separately in \autoref{sec:planar} and \autoref{sec:sphere}.

\begin{theorem}\label{thm:main}
	\begin{enumerate}
		\item Suppose that $\kappa$ is an attenuation coefficient as defined in Definition \ref{deAttCoeff}. Suppose further that the analytic continuation $\tilde\kappa$ is a Nevalinna-Herglotz function (see, for example, (2.19) in \cite{ElbSchShi17}). Then $\tilde\kappa$ is injective on $\overline{\H^+}$.
		\item There exists an analytic function $\kappa^{-1}:\kappa(\overline{\H^+})\to\overline{\H^+}$ serving as the inverse function of $\tilde\kappa$.
        \item Let $\kappa$ be as above, and suppose further that $\kappa^{-1}$ can be analytically continued to $\overline{\H^+}$ (note that the values of $\kappa^{-1}$ on $\H^+\setminus\kappa(\overline{\H^+})$ belongs to the lower half-plane). Let $p(t,x)$ be the unique solution of the (un-attenuated) wave equation \autoref{eqWaveEquation},
		and let $p^a(t,x)$ be the unique solution of the corresponding attenuated wave equation \autoref{eqWaveEquationAttenuation}.
		Then $(\mathcal{F}p)(\omega,x)$ and $(\mathcal{F}p^a)(\omega,x)$ are analytic functions of $\omega$ in $\mathbb{H}^+$.
		\item Moreover, the following relationships between $\Ff_1p$ and $\Ff_1p^a$ holds for all $x\in\mathbb{R}^3$:
		\begin{align}
		(\mathcal{F}p^a)(\omega, x)=\frac{\omega}{\kappa(\omega)}(\mathcal{F}p)(\kappa(\omega), x), &\text{ for all } \omega\in\overline{\mathbb{H}^+}. \label{eq:main:1}\\
		(\mathcal{F}p)(\omega, x)=\frac{\omega}{\kappa^{-1}(\omega)}(\mathcal{F}p^a)(\kappa^{-1}(\omega), x), &\text{ for all } \omega\in\overline{\mathbb{H}^+}. \label{eq:main:2}
		\end{align}
	\end{enumerate}
\end{theorem}
\begin{proof}
\begin{enumerate}
	\item This is a proof that $\tilde\kappa$ is injective on $\R$. We use the Herglotz Representation Theorem to write $\tilde\kappa$ as
	\[
	\tilde\kappa(z)=Az+B+\int_{-\infty}^{\infty}\frac{1+z\nu}{\nu-z}d\sigma(\nu),
	\]
	for some $A>0, B\in \mathbb{R}$, and some bounded increasing function $\sigma:\mathbb{R}\to\mathbb{R}$.
	Then we have
	\[
	\operatorname{Re}\kappa'(z)=A+\int_{-\infty}^{\infty}\frac{(1+\nu^2)\left((\nu-\operatorname{Re}z)^2-(\operatorname{Im}z)^2\right)}{|\nu-z|^4}d\sigma(\nu).
	\]
	
	For $z\in\mathbb{R}$ we have $\operatorname{Im}z=0$, therefore the integrand is non-negative, and therefore $\operatorname{Re}\kappa'(z)\geq A>0$. This means $\operatorname{Re}\kappa(z)$ is monotonic increasing on $\mathbb{R}$, and therefore $\kappa$ is injective on $\R$. To prove that $\tilde\kappa$ is injective on $\overline{\H^+}$, we can use similar arguments on $\kappa'(z)$ and invoke the Noshiro–Warschawski theorem \cite{ChuGev03}. 

	\item This is an direct consequence of (i).

	\item This can be inferred from (ii) and the Paley-Wiener Theorem.

	\item This can be proved by comparing the equations \autoref{eqWaveEquation} and \autoref{eqWaveEquationAttenuation} and using (iii).
\end{enumerate}
\end{proof}

We will now present two examples: the case of a planar observation surface and the case of a spherical observation surface. Let $p^a(t, \vxi)$ for $(t, \vxi)\in \R^+\times \partial\Omega$ be the measurement on the observation surface $\partial\Omega$, and $q^a(t, \vxi)$ be the integrated measurement,
\begin{equation}\label{IntegratedMeasurement}
q^a(t, \vxi)=\int_{-\infty}^t p^a(\tau, \vxi)\d\tau.
\end{equation}
and similarly,
\begin{equation}\label{Integrated-p}
q(t, \vxi)=\int_{-\infty}^t p(\tau, \vxi)\d\tau.
\end{equation}

\begin{theorem}\label{thm:plane}
Suppose that the observation surface $\partial\Omega$ is a plane given by $\R^2\times\{0\}$. $q^a(t, \vxi)\in\S'(\R^+\times\R^2)$ is the integrated measurement which is defined in \autoref{IntegratedMeasurement}, and $\kappa(\omega)$ is attenuation coefficient satisfying \autoref{deAttCoeff} and \autoref{deWeakAttenuation}. If $q^a(t,\vxi)\in\S'_{\kappa^{-1}}(\R^+\times\R^2)$ which is defined in \autoref{IntegratedMeasurement}, then the source term $h(\vx)\in\S'(\R^2\times\R^+)$ with $\vx=(\vx_{12}, \vx_3)$ in \autoref{eqWaveEquationAttenuation} is reconstructed by
\begin{equation}\label{eq:reconstructionPlane}
\Ff h(\sigma,\varrho)=8\pi\i(\Ff q^a)(\kappa^{-1}(\sgn\varrho\sqrt{|\sigma|^2+\varrho^2}), \sigma)
\end{equation}
where $\Ff h(\sigma,\varrho)$ is the Fourier transform of $h(\vx)$ with respect to $\vx_{12}$ and $\vx_{3}$.

\end{theorem}

\begin{theorem}\label{thm:sphere}
Suppose that the observation surface $\partial\Omega$ is a sphere given by $|\vxi|=\varpi$. $q^a(t, \vxi)\in H^1(\R^+\times\partial\Omega)$ is the integrated measurement which is defined in \autoref{IntegratedMeasurement}, and $\kappa(\omega)$ is attenuation coefficient satisfying \autoref{deAttCoeff} and \autoref{deWeakAttenuation}. Then the source term $h(\vx)\in L^2(\Omega_\varepsilon)$ in \autoref{eqWaveEquationAttenuation} is reconstructed by
\begin{equation}\label{eq:reconstructionSphere}
h(\vx)=(I+T)^{-1}\left(-\frac{1}{8\pi^2 \varpi}\Delta_{\vx}\int_{\Gamma}e^{\kappa_{\infty}|\vxi-\vx|}\frac{q^{a}(|\vxi-\vx|, \vxi)}{|\vxi-\vx|}dS(\vxi)\right)
\end{equation}
where $T$ is a compact operator from $L^2(\Omega_\epsilon)$ to $L^2(\Omega_\epsilon)$, and with the kernel
\begin{equation*}
\frac{1}{8\pi^2\varpi}\sum_{j=1}^{\infty}\frac{1}{j!}\Delta_\vx\int_{\partial\Omega}|\vxi-\vy|^{j-1}(\Ff^{-1} (\i\kappa_*(\omega))^j)(|\vxi-\vx|-|\vxi-\vy|)e^{\kappa_\infty(|\vxi-\vx|-|\vxi-\vy|)}\d S(\vxi).
\end{equation*}
\end{theorem}

\begin{remark}
The outline of the two reconstruction formulas can be summarized in the diagram below.
\[
\xymatrix{
h\ar[d]_{R}\ar[r]^{I+T}\ar[rd]^{R^a}&h^a\\
q\ar[r]_{A}&q^a\ar[u]_{B_pM}
}
\]
Here the formula \autoref{eq:reconstructionPlane} follows the lower-left route $q^a\to q\to h$, while the formula \autoref{eq:reconstructionSphere} follows the upper-right route $q^a\to h^a\to h$. The operators in this diagram will be defined in \autoref{sec:planar} and \autoref{sec:sphere}.
\end{remark}

\section{For a plane observation surface}\label{sec:planar}

In this section, we will demonstrate the viability of the reconstruction formula for a planar observation surface. Let $\Omega=\R^2\times \R^{+}\subset \R^3$ be a half-space, and then $\partial\Omega=\R^2\times\{0\}$ is the plane of observation.

\subsection{Preliminaries and notations}\label{preliminaries}

\begin{definition}\label{SchwartzSpace}
The \emph{Schwartz space} on $\R^n$ is the space consisting of smooth functions all of whose derivatives are rapidly decreasing, that is,
\begin{equation*}
\mathcal{S}(\R^n) = \{f\in C^{\infty}(\R^n): \|f\|_{\alpha, \beta}<\infty, \forall \alpha, \beta\in Z_+^n \},
\end{equation*}
where $\alpha, \beta$ are multi-indices, $C^{\infty}(\R^n)$ is the set of smooth functions from $\R^n$ to $\C$ where $\C$ is the complex plane, and
\begin{equation*}
\|f\|_{\alpha, \beta}=\sup_{\vx\in \R^n}|\vx^\beta D^\alpha f(\vx)|
\end{equation*}
are a family of semi-norms on the space.
\end{definition}

\begin{definition}\label{DistributionSpace}
The \emph{space of all tempered distribution} on $\R^n$, denoted by $\mathcal{S}'(\R^n)$, is defined as the dual space of the Schwarz space $\mathcal{S}(\R^n)$. That is, it is the set of all continuous linear functionals
\begin{equation*}
f: \mathcal{S}(\R^n)\to \C,
\end{equation*}
where $\C$ denotes the complex plane.
\end{definition}

For a closed set $D\subset \R^n$, the spaces $\S(D)$ and $\S'(D)$ denotes the Schwartz functions and the tempered distributions with $\emph{support}$ inside $D$, respectively. We emphasize that any subspace of $\S'(\R^n)$ is able to act on any subspace of $\S(\R^n)$.

It is conventional to write $\langle f, \varphi\rangle$ for the result of the distribution $f\in S'(\R^n)$ acting on $\varphi\in S(\R^n)$. In this notation,
\begin{itemize}
\item $f$ is \emph{linear} means that $\langle f, a\varphi+b\psi\rangle = a \langle f, \varphi\rangle + b \langle f, \psi\rangle$ for all $\varphi, \psi \in S(\R^n)$ and all $a, b\in \C$.
\item $f$ is \emph{continuous} means that if $\varphi=\lim_{n\to\infty}\varphi_n$ in $S(\R)$, then $\langle f, \varphi\rangle = \lim_{n\to\infty}\langle f, \varphi_n\rangle$.
\item In this paper, we use the \emph{weak topology} on the space of tempered distributions. That is, $f_n\to f$ in $\S'(\R)$ means that $\lim_{n\to\infty}\langle f_n,\varphi\rangle=\langle f,\varphi\rangle$ for all $\varphi\in\S(\R)$
\item Operator $A: S\to S'$ is continuous means that if $\psi=\lim_{n\to\infty}\psi_n$ in $S(\R)$, then
\begin{equation}
\langle A\psi, \varphi\rangle = \lim_{n\to\infty}\langle A\psi_n, \varphi\rangle
\end{equation}
\item A locally integrable function $g\in L^1_{loc}(\R^n)$ can be identified with a distribution in $\mathcal{S}'(\R^n)$ by considering the inner product
\begin{equation*}
  \langle g,\varphi\rangle=\int_{\R^n}g(\vx)\bar{\varphi}(\vx)\d \vx.
\end{equation*}
\item A distribution $f\in\S'(\R^n)$ is \emph{regular} in an open region $U$ if there exists a locally integrable function $f_0(\vx)$ on $U$, such that
\begin{equation*}
  \langle f,\varphi\rangle=\int_{U}f_0(\vx)\bar{\varphi}(\vx)\d \vx
\end{equation*}
holds for all $\varphi\in\S(\R^n)$ with support inside $U$.
\end{itemize}

The Fourier transform $\Ff$ is an isomorphism of $\mathcal{S}(\R^n)$, which means $\Ff: \mathcal{S}(\R^n)\to \mathcal{S}(\R^n)$. This definition can be extended to $\mathcal{S'}(\R^n)$ by defining $\langle\Ff f, \varphi\rangle=\langle f,\Ff\varphi\rangle$ for all $\varphi\in\mathcal{S}(\R^n)$. Under this definition, the Fourier transform is also an isomorphism of $\mathcal{S}'(\R^n)$, which means $\Ff: \mathcal{S}'(\R^n)\to \mathcal{S}'(\R^n)$.

Since when the observation surface is a plane (and therefore not compact), the measurements will generally does not belong to spaces such as $L^2$. Therefore, we will define the relevant operators on the space of tempered distributions.

\begin{definition}\label{HankelTransform}
The \emph{Hankel transform} of order $\nu$ of a function $f(r)$ is given by
\begin{equation*}
F_{\nu}(k)=\int_0^\infty f(r)J_\nu(kr)r\d r
\end{equation*}
where $J_{\nu}$ is the Bessel function of the first kind of order $\nu$ with $\nu\geqslant -\frac{1}{2}$.
\end{definition}

There is an important relation between Hankel transform and Fourier transform: If the function $f(r)$ is radially symmetric with respect to $r$, then its Fourier transform is also radially symmetric, which is given by
\begin{equation}\label{Hankel-Fourier}
k^{n/2-1}(\Ff f)(k)=(2\pi)^{n/2}\int_0^\infty r^{n/2-1}f(r)J_{n/2-1}(kr)r \d r,
\end{equation}
where $n$ is the dimension of space.

\subsection{Direct problem}

For a vector in $\vx \in\R^2\times\R^+$, we will write it as $\vx=(\vx_{12}, \vx_3)$ where $\vx_{12}\in\R^2$ and $\vx_3\in\R^+$. Since $\partial\Omega$ is a plane $\R^2\times\{0\}$, then for $\vxi\in\partial\Omega$, we slightly abuse the notation to use $\vxi$ to denote its first two components for convenience.

Define the \emph{integrated photoacoustic operator} $R$ by
\begin{equation}\label{R}
R: \mathcal{S}(\R^+\times \R^2)\to \mathcal{S}'(\R^+\times \R^2),\; (Rh)(t, \vxi) = q(t, \vxi).
\end{equation}
where $q(t, \vxi)$ is defined in \autoref{Integrated-p}. It is well-known that the operator $R$ is equal to the spherical mean operator in $\R^3$ (see, for example, \cite{KucKun08}).

We also define the \emph{attenuated integrated photoacoustic operator} by
\begin{equation}\label{Ra}
R^a: \mathcal{S}(\R^+\times \R^2)\to \mathcal{S}'(\R^+\times \R^2),\; (R^ah)(t, \vxi) = q^a(t, \vxi).
\end{equation}
where $q^a(t, \vxi)$ is defined in \autoref{IntegratedMeasurement}.

In this section, we will first give expressions of the operators $R$ and $R^a$ which act on the Schwartz space $\mathcal{S}(\R^+\times \R^2)$. Then we extend their domain to the space of tempered distribution $\mathcal{S}'(\R^+\times \R^2)$ by considering their adjoint operators $R^*$ and $(R^{a})^*$. After this, we will give a formula for the inverse of the adjoint operators, and use that to construct an inverse of the operator $R^a$.

Using the formula for the solution of standard wave equation \autoref{freqsol-noatt}, $R$ is written by the following theorem.
\begin{theorem}
For any $h(\vx)\in\mathcal{S}(\R^3)$, $R$ is defined in \autoref{R}, then
\begin{equation}\label{R-represent}
(Rh)(t, \vxi)= \frac{1}{4\pi}\int_{S^2} t h(\vxi+t\eta)\d S(\eta)
\end{equation}
where $\vxi\in\partial\Omega$ and $\eta\in S^2$ which denotes the unit sphere.
\end{theorem}



\begin{proof}
The solution of the standard wave equation \autoref{eqWaveEquation} is given in \autoref{freqsol-noatt}, by doing the inverse Fourier transform with respect to $\omega$,
\begin{equation*}
p(t, \vxi)= \Ff_1^{-1} \Ff_1 p(\omega,  \vxi)= -\sqrt{2\pi}\int_{\R}\int_{\R^3} \frac{e^{-i\omega t}}{4\pi\sqrt{2\pi}}\frac{\i \omega e^{\i\omega|\vxi-\vy|}}{|\vxi-\vy|}h(\vy)\d \vy\d \omega
\end{equation*}

Therefore, the integrated measurement is written by $q(t, \vxi)=\int_{\R}\int_{\R^3} \frac{e^{-i\omega t}}{4\pi}\frac{e^{\i\omega|\vxi-\vy|}}{|\vxi-\vy|}h(\vy)\d \vy\d \omega$. We consider the integral with respect to $\omega$ first,
\begin{align*}
q(t, \vxi) &= -\frac{1}{4\pi}\int_{\R^3}\frac{1}{|\vxi-\vy|}\int_{\R}e^{i\omega (|\vxi-\vy|-t)}\d\omega h(\vy)\d \vy\\
&= -\frac{1}{4\pi}\int_{\R^3}\frac{\delta(|\vxi-\vy|-t)}{|\vxi-\vy|}h(\vy)\d S(\vy)\\
&= \frac{1}{4\pi t}\int_{|\vxi-\vy|=t}h(\vy)\d \vy.
\end{align*}

Since $\vy$ is on the sphere surface $|\vxi-\vy|=t$, we change the variables $\vy=\vxi+t\eta$ where $\eta$ belongs to the unit sphere $S^2$, with area element $\d S(\vy) = t^2 \d S(\eta)$, and then $Rh$ has the following representation in the time domain,
\begin{align*}
(Rh)(t, \vxi) &= q(t, \vxi)\\
             &= \frac{1}{4\pi}\int_{S^2} t h(\vxi+t\eta)\d S(\eta).
\end{align*}

From this integral representation, we can immediately see that for any $h\in \S(\R^3)$, $(Rh)(t, \vxi)$ is a smooth function bounded above by $t\|h\|_{C^\infty}$, and therefore belongs to $\S'(\R^+\times\R^2)$.

\end{proof}


Sadly, we do not have such elegant representation for the operator $R^a$. By comparing the formulas \autoref{freqsol-att} and \autoref{freqsol-noatt}, the distribution $q^a$ can be obtained from $q$ by, informally speaking, "replacing $\omega$ by $\kappa(\omega)$ in the frequency domain". However, since $\kappa(\omega)\in\C$, we need to define the meaning of having a complex value in the frequency domain, where we use the \emph{Fourier-Laplace transform}, as defined in the appendix, to arrive at the following theorems. We will use the same notation for the Fourier transform and the Fourier-Laplace transform without confusion.

\begin{theorem}\label{Thm-A}
For any $g(t, \xi)\in S'(\R^+\times \R^2)$, $\varphi(t, \xi)\in S(\R^+\times \R^2)$, there exists a continuous operator $A: S'(\R^+\times \R^2)\to S'(\R^+\times \R^2)$ such that
\begin{equation}\label{Operator-A}
\langle \Ff(Ag)(z, \sigma), \Ff \varphi(z, \sigma)\rangle = \langle \Ff g(\kappa(z), \sigma), \Ff \varphi(z, \sigma)\rangle,
\end{equation}
where $z\in \H^+$ and $\sigma$ are the variables corresponding to $t$ and $\vxi$ respectively after the Fourier-Laplace transform.
\end{theorem}

\begin{proof}
We will first prove $Ag\in S'(\R^+\times \R^2)$ for any $g\in S'(\R^+\times \R^2)$. According to \autoref{Hormander}, there exist constants $N'\in \N$ and $C>0$, such that for every $\eta$ with $\Im z \geq \eta$ and $z\in \H^+$,
\begin{equation*}
|\Ff g(z, \sigma)| \leq C(1+\sqrt{|z|^2+| \sigma|^2})^{N'}.
\end{equation*}
Therefore, $\Ff g(z, \sigma)$ is holomorphic with respect to $z$ in the region $\H^+\times\R^2$. Since $\kappa(z)$ is holomorphic from \autoref{deAttCoeff}(ii), $\Ff g(\kappa(z), \sigma)$ is also holomorphic with respect to $z$.

Also, using \autoref{deAttCoeff}(ii), $\Ff g(\kappa(z), \sigma)$ could be bounded by,
\begin{equation}\label{estimate-DirectPbm-OperatorA}
\begin{split}
|\Ff g(\kappa(z), \sigma)|&\leq C(1+\sqrt{|\kappa(z)|^2+| \sigma|^2})^{N'}\\
&\leq C(1+\sqrt{|(1+|z|)^N|^2+|\sigma|^2})^{N'}\\
&\leq C(1+\sqrt{|z|^2+|\sigma|^2})^{NN'}
\end{split}
\end{equation}
for all $z\in\H^+$.

Therefore, applying \autoref{Hormander} on \autoref{estimate-DirectPbm-OperatorA} implies that $\Ff g(\kappa(z), \sigma)$ is the Fourier transform of a tempered distribution supported on $\R^+\times \R^2$, so the operator $A$ is well-defined.

To prove the continuity of $A$, let $g_n\to g$ in $S'(\R^+\times \R^2)$, which means that $\lim_{n\rightarrow \infty} \langle g_n,\phi\rangle=\langle g,\phi\rangle$ for all $\phi\in\mathcal{S}(\R^+\times \R^2)$. For any $\psi\in \S(\R^2)$, in the inner product$\langle g,\phi\rangle$ by taking $\phi(t, \vxi)=\tilde{e}_z(t)\psi(\vxi)$, where $\tilde{e_{z}}$ is defined in the Appendix, we have $\Ff g_n (z, \sigma)= \langle g_n(t, \vxi), \tilde{e}_z(t)\psi(\vxi)\rangle$. Therefore, $\Ff g_n(z, \sigma)$ converges pointwise to $\Ff g(z, \sigma)$ in $\H^+\times\R^2$. For every $\phi\in\mathcal{S}(\R^+\times \R^2)$, we use the dominated convergence theorem to get
\begin{equation*}
\begin{split}
\lim_{n\to+\infty}\langle Ag_n,\phi\rangle &= \lim_{n\to+\infty}\int_{\R^+\times\R^2}\Ff g_n(\kappa(z), \sigma)\overline{\Ff\phi(z, \sigma)}\d z\d\sigma\\
&= \int_{\R^+\times\R^2}\left(\lim_{n\to+\infty}\Ff g_n(\kappa(z), \sigma)\overline{\Ff\phi(z, \sigma)}\right)\d z\d\sigma\\
&=\int_{\R^+\times\R^2}\Ff g(\kappa(z), \sigma)\overline{\Ff\phi(z, \sigma)}\d z\d\sigma\\
&=\langle Ag,\phi\rangle.
\end{split}
\end{equation*}
This shows that $Ag_n\to Ag$ whenever $g_n\to g$, so $A$ is continuous.

Using the operator $A$, the relationship between the operators $R$ and $R^a$ will be given in the following theorem.

\begin{theorem}\label{relation-Ra-R-A}
For $q(t, \vxi)$ and $q^a(t, \vxi)$ defined in \autoref{Integrated-p} and \autoref{IntegratedMeasurement} respectively, operator $A$ defined in \autoref{Operator-A} satisfies
\begin{equation*}
Aq(t, \vxi) = q^a(t, \vxi).
\end{equation*}
Furthermore, for all $h(\vx)\in\mathcal{S}(\R^3)$,
\begin{equation}
R^ah(\vx)=ARh(\vx).
\end{equation}
\end{theorem}

\begin{proof}

Combining the definition of $q(t, \vxi)$ in \autoref{Integrated-p} and the representation of $\Ff_1 p(\omega, \vxi)$ in \autoref{freqsol-noatt} yields
\begin{equation}\label{Fourier-q-R}
\Ff_1q(\omega, \vxi)=-\int_{\R^3} \frac{1}{4\pi\sqrt{2\pi}}\frac{\e^{\i\omega|\vxi-\vy|}}{|\vxi-\vy|}h(\vy)\d \vy
\end{equation}
for all $\omega\in\R$.

We will prove that \autoref{Fourier-q-R} also holds when $\omega$ is replaced by $z\in \H^+$, by showing that both sides are holomorphic functions with respect to $z$.

On one hand, since $\Ff_1 q(z, \vxi)$ is the Fourier-Laplace transform of $q$, \autoref{Hormander} shows that $\Ff_1q(z, \vxi)$ is holomorphic with respect to $z$. On the other hand, to prove the right hand of \autoref{Fourier-q-R} is also holomorphic with respect to $z$, we verify the Cauchy-Riemann equations.

\begin{equation*}
\begin{split}
&\frac{\partial}{\partial \bar{z}}\left(-\int_{\R^3} \frac{1}{4\pi\sqrt{2\pi}}\frac{\e^{\i z|\vxi-\vy|}}{|\vxi-\vy|}h(\vy)\d \vy\right)\\
&=-\int_{\R^3} \frac{1}{4\pi\sqrt{2\pi}}\frac{\partial \e^{\i z|\vxi-\vy|}}{\partial \bar{z}}\frac{1}{|\vxi-\vy|}h(\vy)\d \vy\\
&=0.
\end{split}
\end{equation*}
where the last equality holds by using the fact that $\e^{\i z|\vxi-\vy|}$ is holomorphic with respect to $z$, so $\frac{\partial \e^{\i z|\vxi-\vy|}}{\partial \bar{z}}=0$.

Since the both sides of \autoref{Fourier-q-R} agree on the real axis $\R$, by the identity theorem for holomorphic functions, they must agree everywhere in $\overline{\H^+}$. Based on this, we can substitute $\kappa(z)$ for $z$ in \autoref{Fourier-q-R}. Using the definition of the operator $A$, for any $\varphi(t, \xi)\in S(\R^+\times \R^2)$ the left hand becomes
\begin{align*}
\langle (\Ff_1 q)(\kappa(z), \vxi), (\Ff_1 \varphi)(z, \vxi)\rangle&= \langle(\Ff q)(\kappa(z), \sigma), (\Ff \varphi)(z, \sigma)\rangle\\
&= \langle(\Ff Aq)(z, \sigma), (\Ff \varphi)(z, \sigma)\rangle\\
&= \langle(\Ff_1 Aq)(z, \vxi), (\Ff_1 \varphi)(z, \vxi)\rangle.
\end{align*}

The right hand side becomes
\begin{equation*}
-\int_{\R^3} \frac{1}{4\pi\sqrt{2\pi}}\frac{\e^{\i\kappa(z)|\vxi-\vy|}}{|\vxi-\vy|}h(\vy)\d \vy= \Ff_1 q^a(z, \vxi).
\end{equation*}

This shows $Aq=q^a$ since Fourier transform is a bijection. Consequently, for any $h(\vx)\in\mathcal{S}(\R^2\times\R^+)$,
\begin{equation*}
ARh=Aq=q^a=R^ah,
\end{equation*}
which impiles $R^a=AR$.

\end{proof}

\subsection{Inverse problem}

To invert the operator $R^a$, \autoref{relation-Ra-R-A} shows it is necessary to know how to invert operators $R$ and $A$ . It turns out both inversions can be done in frequency space, and the inverse operators can be composed naturally. We will first discuss how to invert the operator $R$ using its adjoint, and then directly construct the inverse operator of $A$.

\subsubsection{The adjoint of the operator $R$}

In this part, we will prove that $R$ is a continuous operator by constructing its adjoint operator $R^*$, and extend its domain to the space $\mathcal{S}'(\R^+\times \R^2)$. This allows us to invert $R$ via the inverse of its adjoint.

Define
\begin{equation}\label{Space_Se}
\S_{e}(\R \times \R^2)=\{g\in \S(\R \times \R^2): g(\omega, \sigma)=g(-\omega, \sigma)\;\; \forall\; \omega\in\R, \;\sigma\in\R^2\}.
\end{equation}

\begin{theorem}\label{Thm-R*-representation}
There exists a continuous linear operator $R^*:\mathcal{S}(\R^+\times \R^2)\to \mathcal{S}_e(\R^3)$ such that
$\langle Rh, \varphi\rangle=\langle h, R^*\varphi\rangle$ for all $h(\vx)\in\mathcal{S}(\R^2\times \R^+)$ and $\varphi(t, \vxi)\in\mathcal{S}(\R^+\times \R^2)$. Furthermore, the operator $R^a$ in the time domain is given by
\begin{equation}\label{Rstar_time}
(R^*\varphi)(\vx)= \frac{1}{4\pi}\int_{\R^2}\varphi(|\vx-\vxi|, \vxi)|\vx-\vxi|^{-1}\d \vxi,
\end{equation}
where $\vxi\in \partial\Omega$ and $\vx\in\R^3$. In frequency domain $R^*$ is given by
\begin{equation}\label{Rstar-frequency}
(\Ff R^*\varphi)(\sigma, \varrho)= \frac{1}{8\pi\i}\frac{(\Ff \varphi(\sqrt{|\sigma|^2+\varrho^2}, \sigma)- \Ff \varphi(-\sqrt{|\sigma|^2+\varrho^2}, \sigma))}{\sqrt{|\sigma|^2+\varrho^2}},
\end{equation}
where $\sigma$ and $\varrho$ are the variables corresponding to $\vx_{12}$ and $\vx_3$ after the Fourier transform.
\end{theorem}

\begin{proof}
For any $\varphi(t, \vxi)\in S(\R^+, \R^2)$,
\begin{equation*}
\langle Rh, \varphi\rangle = \int_{\R^+\times\R^2}(Rh)(t, \vxi)\overline{\varphi(t, \vxi)}\d \vxi \d t.
\end{equation*}
Plugging \autoref{R-represent} into it yields $\langle Rh, \varphi\rangle = \frac{1}{4\pi\sqrt{2\pi}}\int_{\R^+\times\R^2}\int_{S^2} t h(\vxi+t\eta)\overline{\varphi(t, \vxi)}\d S(\eta)\d \vxi \d t$. We make a substitution $\vx=t\eta$ for $t\in \R$ and $\eta\in \R^3$ so that $\d\vx=t^2 \d t\d S(\eta)$. Plugging this relation to the integral and then do a translation, the inner product becomes
\begin{align*}
\langle Rh, \varphi\rangle &= \frac{1}{4\pi}\int_{\R^3}\int_{\R^2}  |\vx|^{-1} h(\vxi+\vx)\overline{\varphi(|\vx|, \vxi)}\d \vxi \d \vx\\
&= \frac{1}{4\pi}\int_{\R^3}\int_{\R^2} h(\vx)\overline{\varphi(|\vx-\vxi|, \vxi)}|\vx-\vxi|^{-1}\d \vxi \d \vx.
\end{align*}

Therefore, the adjoint operator $R^*$ is given by
\begin{equation*}
(R^*\varphi)(\vx)= \frac{1}{4\pi}\int_{\R^+\times\R^2}\int_{\R^2}\varphi(|\vx-\vxi|, \vxi)|\vx-\vxi|^{-1}\d \vxi
\end{equation*}

so that $\langle Rh, \varphi\rangle=\langle h, R^*\varphi\rangle$ always holds.

The computation is much faster if the formula is in the frequency domain as mentioned \cite{HalSchZan09b}, the representation of $R^*$ in the frequency domain is needed.

Applying Fourier transform with respect to $\vx=(\vx_{12}, \vx_3)$ in \autoref{Rstar_time}, and let $\sigma\in \R^2$ and $\varrho\in \R$ be the variables after the Fourier transform corresponding to $\vx_{12}$ and $\vx_3$ respectively,

\begin{equation*}
(\Ff {R^*\varphi})(\sigma, \varrho) = \frac{1}{4\pi(2\pi)^{3/2}}\int_{\R^3} e^{\i \langle \sigma, \vx_{12} \rangle+\i \varrho\vx_3}\int_{\R^2}\frac{\varphi(\sqrt{|\vx_{12}-\vxi|^2+\vx_3^2}, \vxi)}{\sqrt{|\vx_{12}-\vxi|^2+\vx_3^2}} \d\vxi\d\vx_{12}\d\vx_3.
\end{equation*}

By doing a translation $\vx_{12}\to\vx_{12}+\vxi$, and first consider the integral with respect to $\vxi$,
\begin{align*}
(\Ff{R^*\varphi})(\sigma, \varrho)&= \frac{1}{4\pi(2\pi)^{3/2}}\int_{\R^3\times\R^2} e^{\i \langle \sigma, \mathbf{\vx_{12}} \rangle}e^{\i \langle \sigma, \vxi \rangle}e^{\i \varrho\vx_3}\frac{\varphi(|\vx|, \vxi)}{|\vx|} \d\vx \d\vxi\\
&= \frac{1}{4\pi\sqrt{2\pi}}\int_{\R^3} e^{\i \langle \sigma, \mathbf{\vx_{12}} \rangle + \i \varrho\vx_3}\frac{\Ff_2 \varphi(|\vx|, \sigma)}{|\vx|} \d\vx\\
&= \frac{1}{4\pi\sqrt{2\pi}}\int_{\R^3} e^{\i \langle (\sigma, \varrho), \vx \rangle} \frac{\Ff_2 \varphi(|\vx|, \sigma)}{|\vx|} \d\vx,
\end{align*}
which could be seen as Fourier transform with respect to $\vx$. By taking $|\vx|=\nu$, $\frac{\Ff_2 \varphi(\nu, \sigma)}{\nu}$ is a radially symmetric function with respect to $\nu$. Inserting the relation between Hankel transform and Fourier transform \autoref{Hankel-Fourier} into above representation of $(\Ff{R^*\varphi})(\sigma, \varrho)$ yields
\begin{equation}\label{R*-varphi}
\Ff (R^*\varphi)(\sigma, \varrho)=\frac{1}{4\pi\sqrt{2\pi}|\langle \sigma, \varrho \rangle|^{1/2}}\int_{\R^+} \nu^{3/2} J_{1/2}(\nu |\langle \sigma, \varrho \rangle|)\frac{\Ff_2 \varphi(\nu, \sigma)}{\nu} \d \nu.
\end{equation}

Bessel function $J_{1/2}$ has elementary representation, which is $J_{1/2}(z)=\frac{\sin z}{\sqrt{z}}$. Inserting it into \autoref{R*-varphi} yields
\begin{align*}
(\Ff (R^*\varphi))(\sigma, \varrho)&= \frac{1}{4\pi\sqrt{2\pi}|\langle \sigma, \varrho \rangle|^{1/2}}\int_{\R^+} \nu^{3/2} \frac{\sin(\nu |(\sigma, \varrho)|)}{\nu^{1/2}|(\sigma, \varrho)^{1/2}|}\frac{\Ff_2 \varphi(\nu, \sigma)}{\nu} \d \nu\\
&= \frac{1}{4\pi\sqrt{2\pi}|\langle \sigma, \varrho \rangle|}\int_{\R^+} \sin(\nu|(\sigma, \varrho)|)\Ff_2 \varphi(\nu, \sigma) \d \nu
\end{align*}

Since $|\langle \sigma, \varrho \rangle|= \sqrt{|\sigma|^2+\varrho^2}$, and using Euler's formula for $\sin$ function,
\begin{equation*}
(\Ff (R^*\varphi))(\sigma, \varrho)= \frac{1}{8\pi i\sqrt{2\pi}|\langle \sigma, \varrho \rangle|}(\int_{\R^+}e^{\i \nu \sqrt{|\sigma|^2+\varrho^2}}\Ff_2 \varphi(\nu, \sigma) \d \nu- \int_{\R^+}e^{-\i \nu \sqrt{|\sigma|^2+\varrho^2}}\Ff_2 \varphi(\nu, \sigma) \d \nu).
\end{equation*}

Using the definition of Fourier transform, in the frequency domain $R^*$ is written by,

\begin{equation}
\Ff (R^*\varphi)(\sigma, \varrho)= \frac{1}{8\pi i}\frac{(\Ff \varphi(\sqrt{|\sigma|^2+\varrho^2}, \sigma)- \Ff\varphi(-\sqrt{|\sigma|^2+\varrho^2}, \sigma))}{\sqrt{|\sigma|^2+\varrho^2}}.
\end{equation}

Since \autoref{R} shows $R: \S\to \S'$, naturally, $R^*$ is an operator mapping $\S$ to $\S'$. We will now prove that the range of $R^*$ actually lies in $\S$. For any $\varphi\in \S$, we will prove $R^*\varphi\in S$ by showing all its semi-norms $\|\Ff(R^*\varphi)\|_{\alpha,\beta}$ are finite.

First of all, to get rid of the denominator $\sqrt{|\sigma|^2+\varrho^2}$, we write
\begin{equation*}
\Ff(R^*\varphi)=\int_{-1}^{1}D^{(1,0)}\left(\Ff\varphi(\zeta\sqrt{|\sigma|^2+\varrho^2},\sigma)\right)\d\zeta,
\end{equation*}
where $D^{(1,0)}$ denoted the defferential operator with respect to the first variable with the order $1$.

Using the chain rule repeatedly, by induction there exists a family of polynomials $P_{\alpha,\alpha'}$ for all multi-indices $\alpha,\alpha'$ with $|\alpha'|\leq|\alpha|$, such that the derivative with order $\alpha$
\begin{equation*}
D^\alpha_{\sigma,\varrho}\left(\Ff\varphi(\zeta\sqrt{|\sigma|^2+\varrho^2},\sigma)\right)=
\sum_{|\alpha'|\leq|\alpha|}P_{\alpha,\alpha'}\left(\frac{\zeta\varrho}{\sqrt{|\sigma|^2+\varrho^2}},\frac{\sigma}{\sqrt{|\sigma|^2+\varrho^2}}\right)D^{\alpha'}\left(\Ff\varphi(\zeta\sqrt{|\sigma|^2+\varrho^2},\sigma)\right),
\end{equation*}
where $D^\alpha_{\sigma, \varrho}$ denotes the $\alpha$-th derivative with respect to $\sigma$ and $\varrho$. Therefore, the semi-norms $\|\Ff(R^*\varphi)\|_{\alpha,\beta}$ can be controlled by the semi-norms $\|\Ff\varphi\|_{\alpha',\beta}$ for all $|\alpha'|\leq|\alpha|$. Consequently, $R^*\varphi$ does indeed belong to $\mathcal{S}(\R^+\times\R^2)$.

The above argument also shows that the operator $R^*$ is continuous, since when $\varphi\to 0$ in $\mathcal{S}(\R^+\times\R^2)$, all its semi-norms tends to zero, which means that $\R^*\varphi$ also tends to zero.
\end{proof}

\begin{lemma}
The mapping $R: \S\to \S'$ could be extended to a continuous linear operator $R: \S'\to \S'$.
\end{lemma}
\begin{proof}

\autoref{Thm-R*-representation} shows that $R^*$ is an operator which maps $\S$ to $\S$, we now extend operator $R$ such that $Rh\in\S'(\R^+\times\R^2)$ for all $h\in\S'(\R^+\times\R^2)$ by
\begin{equation*}
\langle Rh,\varphi\rangle=\langle h,R^*\varphi\rangle \text{ for all } \varphi\in\mathcal{S}(\R^+\times\R^2).
\end{equation*}
For any $h\in \S'$, $Rh$ is a well-defined linear functional, and satisfies $Rh(\varphi)=(h\circ R^*) (\varphi)$. Since $h$ is a continuous linear functional and $R^*$ is a continuous operator, the linear functional $Rh$ is also continuous. Therefore, $Rh$ belongs to $\mathcal{S}'(\R^+\times\R^2)$.

The continuity of the extended operator $R$ can be proved by the following. For any sequence $\{h_n\}_n$ in $\S'(\R^+\times\R^2)$ satisfying $h_n\to h$, which means that $\lim_{n\to \infty}\langle h_n, \varphi\rangle =\langle h, \varphi\rangle$ holds for any $\varphi \in \S(\R^+\times\R^2)$. Since $R^*\varphi\in \S$ from \autoref{Thm-R*-representation},
\begin{equation*}
\begin{split}
\lim_{n\to \infty}\langle Rh_n,\varphi\rangle&= \lim_{n\to \infty}\langle h_n, R^*\varphi\rangle\\
= \langle h,R^*\varphi\rangle\\
= \langle Rh,\varphi\rangle.
\end{split}
\end{equation*}

Therefore, the limit of $Rh_n$ is $Rh$, which means $R$ is continuous.

\end{proof}

\subsubsection{Inverting the operator $R$}
In this part we will give the inverse operator of $R$ by constructing a continuous operator that serves as an inverse of $R^*$. To construct this operator, in \autoref{E2thm} we will use a theorem in \cite{See64} which is similar to Whitney's extension theorem.

\begin{lemma}\label{Ethm}
There exists a continuous linear mapping
\begin{equation}
E: \mathcal{S}_e(\R^3) \to \S_e(\R^3)
\end{equation}
such that for all $\psi\in \S_e(\R^3)$ where $\S_e(\R^3)$ is defined in \autoref{Space_Se}, in the frequency domain $E\psi$ satisfies
\begin{equation}\label{E-operator}
(\Ff(E\psi))(\omega, \sigma)= 8\pi\i\omega (\Ff\psi)(\sgn (\omega)\sqrt{|\omega|^2-|\sigma|^2}, \sigma),  \text{ for } |\omega|\geq |\sigma|.
\end{equation}
\end{lemma}

\begin{proof}

We will construct continuous linear operators $E_1$ and $E_2$ such that $E=\Ff^{-1}\circ E_1\circ E_2\circ \Ff$. Since the Fourier transform $\Ff$ is continuous, $E$ would also be continuous.

\begin{lemma}
There exists a continuous operator $E_1:\mathcal{S}(\R^3)\to \mathcal{S}_e(\R^3)$ such that
\begin{equation}
E_1\psi(\omega, \sigma)=8\pi\i\psi(\omega^2-\sigma^2, \sigma),  \text{ for all } \omega\in\R.
\end{equation}
holds for all $\psi\in \mathcal{S}(\R^3)$.
\end{lemma}
\begin{proof}
Since the partial derivatives of $E_1\psi(\omega, \sigma)=\psi(\omega^2-\sigma^2, \sigma)$ are combinations of polynomials of $\omega$, $\sigma$ as well as the derivatives of $\psi(\omega, \sigma)$, it follows that $E_1\psi$ is a Schwartz function, and all its semi-norms $\|E_1\psi\|_{\alpha, \beta}$ can be controlled by the semi-norms of $\psi$. This means the operator $E_1$ is continuous.
\end{proof}

\begin{lemma}\label{E2thm}
There exists a continuous operator $E_2:\mathcal{S}_e(\R^3)\to \mathcal{S}(\R^3)$ such that, for all $\psi\in\mathcal{S}_e(\R^3)$,
\begin{equation}
E_2\psi(\omega, \sigma)=\omega\psi(\sqrt{\omega}, \sigma)
\end{equation}
holds for all $\omega\geq 0$.
\end{lemma}
\begin{proof}
Define $$\chi(\omega, \sigma)=\psi(\sqrt{\omega}, \sigma)$$ for all $\omega\geq 0$. Since $\psi\in\mathcal{S}_e(\R^3)$, it is obvious that $\chi$ is smooth when $\omega>0$ and rapidly decreasing. We will first prove that $\frac{\partial^n}{\partial\omega^n}\chi(\omega, \sigma)$ are smooth functions when $\omega=0$. (directly using the chain rule could not help since the derivatives of $\sqrt{\omega}$ are not bounded when $\omega$ is small.)
		
Since $\psi$ is a smooth, by doing a Taylor expansion for $\psi(\omega, \sigma)$ at $\omega=0$, and noticing that $\psi$ is even with respect to $\omega$, so $\psi$ could be written by
\begin{equation}
\psi(\omega, \sigma)=\sum_{j=0}^{k}\psi_{2j}(\sigma)\frac{\omega^{2j}}{(2j)!}+o(|\omega|^{2k}),
\end{equation}
where $\psi_{2j}(\sigma)$ is smooth functions on $\R^2$. This implies the function $\chi$ also has the Taylor expansion when $\omega\to 0^+$,
\begin{equation}
\chi(\omega, \sigma)=\sum_{j=0}^{n}\psi_{2j}(\sigma)\frac{\omega^{j}}{(2j)!}+o(|\omega|^{n}).
\end{equation}
for all $\omega\geq 0$. This shows that $\frac{\partial^j}{\partial\omega^j}\chi(0,\sigma)=\frac{j!}{(2j)!}\psi_{2j}(\sigma)$ is indeed a smooth function.

The next step is to do the extension of $\chi$ to $\omega<0$. Following Seelay in \cite{See64}, we choose a sequence of real numbers $\{a_m\}_{m\geq 1}$ such that $\sum_{m\geq 1}a_m2^{m j}=(-1)^j$, and a bump function $\theta(\omega)\in C^\infty_0(\R)$ such that $\theta(\omega)=1$ when $|\omega|<1$. Define
\begin{equation}
\tilde{\chi}(\omega, \sigma)=\left\{ \begin{array}{lr}
\chi(\omega, \sigma), & \omega\geq0\\
\sum_{m=1}^{\infty}a_m\chi(-2^m\omega, \sigma)\theta(-2^m\omega), & \omega<0.
\end{array}\right.
\end{equation}
		
It is obvious that $\tilde{\chi}(\omega, \sigma)$ is a smooth function when $\omega<0$ and rapid-decreasing since the construction of $\theta$ deduces $\tilde{\chi}=0$ when $\omega\rightarrow -\infty$. what remains is to prove the smoothness of $\tilde{\chi}(\omega, \sigma)$ at $\omega=0$. This will be done by computing its Taylor expansion when $\omega\to0^-$.
		
The definition of $\theta(\omega)$ deduces that $\theta(-2^m\omega)=1+o(|\omega|^n)$ for all $n\in Z^+$, and using the definition of sequence $a_m$ yields
\begin{equation}
\begin{split}
\tilde{\chi}(\omega, \sigma)&=\sum_{m=1}^{\infty}a_m\chi(-2^m\omega, \sigma)\theta(-2^m\omega)\\
&=\sum_{m=1}^{\infty}a_m\sum_{j=0}^{n}\psi_{2j}(\sigma)\frac{(-2^m\omega)^{j}}{(2j)!}+o(|\omega|^{n})\\
&=\sum_{j=0}^{n}\left(\sum_{m=1}^{\infty}(-1)^ja_m2^{mj}\right)\psi_{2j}(\sigma)\frac{\omega^{j}}{(2j)!}+o(|\omega|^{n})\\
&=\sum_{j=0}^{n}\psi_{2j}(\sigma)\frac{\omega^{j}}{(2j)!}+o(|\omega|^{n}),
\end{split}
\end{equation}
which agrees with the Taylor expansion when $\omega\to 0^+$. This shows that $\tilde{\chi}$ is smooth at $\omega=0$, and therefore belongs to $\mathcal{S}(\R\times\R^2)$.

We now proceed to estimate the semi-norms of $\tilde{\chi}$ to prove that the operator mapping $\psi$ to $\tilde{\chi}$ is continuous. This will be broken up into two parts: first we estimate the derivatives of $\chi$ by the derivatives of $\psi$, then we control the derivatives of $\tilde{\chi}$ by the derivatives of $\chi$.
		
For the first part, a representation for the derivatives of $\chi$ is needed. The following integral representation  can be proved by induction,
\begin{equation}
\frac{\partial^j}{\partial\omega^j}\chi(\omega,\sigma)=\frac{(-1)^j}{2^{2j-1}(j-1)!}\int_{0}^{1}(1-\zeta^2)^{j-1}\left((\frac{\partial^{2j}}{\partial\omega^{2j}}\psi)(\zeta\sqrt{\omega},\sigma)\right)\d\zeta
\end{equation}
holds for all $\omega>0$ and $j\geq 1$.
		
Therefore, the semi-norms of $\chi$ could be bounded by,
\begin{equation}
\begin{split}
&\sup_{\R^+\times\R^2}|\omega^{\alpha_1}\sigma^{\alpha_2}D_\omega^{\beta_1}D_\sigma^{\beta_2}\chi(\omega, \sigma)|\\
&\leq \sup_{\R^+\times\R^2}C\left|\int^1_0(1-\zeta^2)^{\beta_1-1}\left(\omega^{\alpha_1}\sigma^{\alpha_2}(\frac{\partial^{\beta_2}}{\partial\sigma^{\beta_2}}\frac{\partial^{2\beta_1}}{\partial\omega^{2\beta_1}}\psi)(\zeta\sqrt{\omega},\sigma)\right)\d\zeta\right|\\
&\leq C\left(\int^1_0(1-\zeta^2)^{\beta_1-1}\d\zeta\right)\|\psi\|_{(2\alpha_1,\alpha_2),(2\beta_1,\beta_2)}\\
&\leq C\|\psi\|_{(2\alpha_1,\alpha_2),(2\beta_1,\beta_2)}\\
\end{split}
\end{equation}
where $C$ is a constant that only depends on $\beta_1$.
		
This means that the semi-norms of $\chi$ can be controlled by the corresponding semi-norms of $\psi$.
		
Again using the arguments from \cite{See64}, the semi-norms of $\tilde{\chi}$ can be controlled by the semi-norms of $\chi$. This finally shows that the operator $E_2$ is continuous.
		
\end{proof}
	
It follows that
\begin{equation}
E=\Ff^{-1}\circ E_1\circ E_2\circ \Ff
\end{equation}
is a continuous operator.
	
\end{proof}

The following two lemmas shows that $E$ can be thought as an inverse of $R^*$.

Define
\begin{equation}\label{S'cone}
S'_{\text{cone}}(\R^+\times \R^2)=\{g\in S'(\R^+\times \R^2): \Ff g(\omega, \sigma)=0 \text{ for } |\omega|< |\sigma|\}.
\end{equation}

\begin{lemma}\label{Relation-ER}
For the operator $E$ satisfying \autoref{Ethm} and the operator $R^*$ satisfying \autoref{Thm-R*-representation}, the following relation holds,
\begin{equation}
\langle g, ER^*\varphi\rangle=\langle g,\varphi\rangle
\end{equation}
for all $g\in S'_{\text{cone}}(\R^+\times \R^2)$ and $\varphi\in\mathcal{S}(\R^+\times \R^2)$.
\end{lemma}

\begin{proof}
Let $g\in S'_{\text{cone}}$, by the Parseval equality,
\begin{equation}\label{ER-Identitythem1-step1}
\begin{split}
\langle g,ER^*\varphi\rangle&=\langle\Ff g,\Ff(E R^*\varphi)\rangle\\
&=\int_{\R\times\R^2}\Ff g(\omega, \sigma)\overline{\Ff(E R^*\varphi)(\omega, \sigma)}\d\omega\d\sigma.
\end{split}
\end{equation}
	
By the definition of the space $S'_{\text{cone}}(\R^+\times \R^2)$, the distribution $\Ff g$ is supported in the set $\{(\omega,\sigma)\in\R\times\R^2\mid|\omega|\geq|\sigma|\}$. Therefore, we could restrict the integral $\R\times\R^2$ to $|\omega|\geq|\sigma|$, and then plug \autoref{E-operator} into \autoref{ER-Identitythem1-step1} yields,
\begin{align*}
\int_{\R\times\R^2}\Ff g(\omega, \sigma)\overline{\Ff(E R^*\varphi)(\omega, \sigma)} &=\int_{|\omega|\geq|\sigma|}\Ff g(\omega, \sigma)\overline{\Ff(E R^*\varphi)(\omega, \sigma)}\d\omega\d\sigma\\
&=8\pi\i\int_{|\omega|\geq|\sigma|}\Ff g(\omega, \sigma)\omega\overline{\Ff(R^*\varphi)(\sgn (\omega)\sqrt{|\omega|^2-|\sigma|^2}, \sigma)}\d\omega\d\sigma
\end{align*}
	
Then inserting the expression \autoref{Rstar-frequency} into above, we have
\begin{align*}
&\phantom{=}8\pi\i\int_{|\omega|\geq|\sigma|}\Ff g(\omega, \sigma)\omega\overline{\Ff(R^*\varphi)(\sgn (\omega)\sqrt{|\omega|^2-|\sigma|^2}, \sigma)}\d\omega\d\sigma\\
&=\int_{|\omega|\geq|\sigma|}\Ff g(\omega, \sigma)\sgn\omega\overline{\Ff(\varphi)(|\omega|, \sigma)-\Ff(\varphi)(-|\omega|, \sigma)}\d\omega\d\sigma.
\end{align*}
	
To remove $\sgn \omega$, we divide the integral with respect to $\omega$ on $\R$ into two parts, $\omega\geq|\sigma|$ and $\omega\leq-|\sigma|$, and then
	
\begin{equation*}
\int_{|\omega|\geq|\sigma|}\Ff g(\omega, \sigma)\sgn\omega\overline{\Ff(\varphi)(|\omega|, \sigma)-\Ff(\varphi)(-|\omega|, \sigma)}\d\omega\d\sigma= P + Q,
\end{equation*}
where
\begin{equation*}
P=\int_{\omega\geq|\sigma|}\Ff g(\omega, \sigma)\overline{\Ff(\varphi)(\omega, \sigma)-\Ff(\varphi)(-\omega; \sigma)}\d\omega\d\sigma,
\end{equation*}
and
\begin{equation*}
\begin{split}
Q&=-\int_{\omega\leq-|\sigma|}\Ff g(\omega, \sigma)\overline{\Ff(\varphi)(-\omega, \sigma)-\Ff(\varphi)(\omega, \sigma)}\d\omega\d\sigma\\
&=\int_{\omega\leq-|\sigma|}\Ff g(\omega, \sigma)\overline{\Ff(\varphi)(\omega, \sigma)-\Ff(\varphi)(-\omega, \sigma)}\d\omega\d\sigma,
\end{split}
\end{equation*}

which can be recombined to become

\begin{align*}
P+Q &=\int_{|\omega|\geq|\sigma|}(\Ff g)(\omega, \sigma)\overline{(\Ff\varphi)(\omega, \sigma)-(\Ff\varphi)(-\omega, \sigma)}\d\omega\d\sigma.
\end{align*}
By a change of variable $\omega\mapsto-\omega$, we have $$\int_{|\omega|\geq|\sigma|}(\Ff g)(\omega, \sigma)\overline{(\Ff\varphi)(-\omega, \sigma)}\d\omega\d\sigma=\int_{|\omega|\geq|\sigma|}(\Ff g)(-\omega, \sigma)\overline{(\Ff\varphi)(\omega, \sigma)}\d\omega\d\sigma,$$ and therefore
\begin{equation*}
P+Q =\int_{|\omega|\geq|\sigma|}((\Ff g)(\omega, \sigma)-(\Ff g)(-\omega, \sigma))\overline{(\Ff\varphi)(\omega, \sigma)}\d\omega\d\sigma.
\end{equation*}

Since Fourier transform has the property that $\Ff (g(-t, \sigma))=(\Ff g)(-\omega, \sigma)$, combining this property with Parseval identity,
\begin{align*}
\int_{\R\times\R^2}(\Ff g(\omega, \sigma)-\Ff g(-\omega,\sigma))\overline{\Ff(\varphi)(\omega, \sigma)}\d\omega\d\sigma  &=\int_{\R\times\R^2}(g(t, \vxi)-g(-t, \xi))\overline{\varphi(t, \vxi)}\d t\d\vxi\\
&=\int_{\R^+\times\R^2}g(t, \vxi)\overline{\varphi(t, \vxi)}\d t\d\vxi\\
&=\langle g,\varphi\rangle,
\end{align*}
which finishes the proof.
	
\end{proof}

\begin{theorem}\label{Thm:RstarE}
For every $\varphi\in\S_e(\R\times\R^2)$, we have
\begin{equation}
R^*E\varphi=\varphi.
\end{equation}
\end{theorem}
\begin{proof}
Using the definitions of the operators $R^*$ and $E$,
\begin{equation*}
\begin{split}
(\Ff R^*E\varphi)(\omega, \sigma)&=\frac{1}{8\pi i\sqrt{\omega^2+\sigma^2}}\left((\Ff E\varphi)(\sqrt{\omega^2+\sigma^2}, \sigma)-(\Ff E\varphi)(-\sqrt{\omega^2+\sigma^2}, \sigma)\right)\\
&=\frac12\left((\Ff\varphi)(\omega, \sigma)+(\Ff\varphi)(-\omega, \sigma)\right)\\
&=(\Ff\varphi)(\omega, \sigma)
\end{split}
\end{equation*}
where the last equality is because of the evenness of $\varphi$.
\end{proof}

\end{proof}

\subsubsection{The reconstruction formula}

\begin{theorem}\label{Thm-reconstructionR}
Let $U\subset\R^2\times\R$ be an open region, and $U'\subset\R\times\R^2$ be the image of $U$ under the map $(\sigma,\varrho)\mapsto(\sgn\varrho\sqrt{\varrho^2+|\sigma|^2},\sigma)$. If for any $q(t, \vxi)$, the Fourier transform $(\Ff q)(\omega,\sigma)$ is regular on $U'$, then $(\Ff h)(\sigma,\varrho)$ is regular on $U$, and the following equality holds for all $(\sigma,\varrho)\in U$,
\begin{equation*}
\Ff h(\sigma,\varrho)=8\pi\i\varrho(\Ff q)(\sgn\varrho\sqrt{\varrho^2+|\sigma|^2}, \sigma).
\end{equation*}
\end{theorem}

\begin{proof}
Let $\varphi\in\mathbb{S_e}(\R^3)$ be any test function such that $\Ff\varphi$ is supported in $U$. We claim that $\supp \Ff E\varphi\subset U'$. To prove this claim, let $(\omega, \sigma)\not\in U'$, the definition of $U'$ deduces that $(\sgn\omega\sqrt{\omega^2-\sigma^2}, \sigma)\not\in U$. From \autoref{E-operator}, $\Ff E\varphi(\omega, \sigma)= 8\pi\i\omega (\Ff\psi)(\sgn (\omega)\sqrt{|\omega|^2-|\sigma|^2}, \sigma)= 0$. Thus $\supp \Ff E\varphi\subset U'$. Then
\begin{equation*}
\begin{split}
\langle h,\varphi\rangle&=\langle Rh,E\varphi\rangle\\
&=\int_{U'}(\Ff q)(\omega,\sigma)\overline{(\Ff E\varphi)(\omega,\sigma)}\d\omega\d\sigma\\
&=8\pi\i\int_{U'}(\Ff q)(\omega,\sigma)\overline{\omega(\Ff\varphi)(\sgn\omega\sqrt{\omega^2-|\sigma|^2}, \sigma)}\d\omega\d\sigma.
\end{split}
\end{equation*}
Making a variable substitution $\omega=\sgn\varrho\sqrt{\varrho^2+|\sigma|^2}$ so that $\d\omega=\frac{\varrho}{\omega}\d\varrho$, and then
\begin{equation*}
\begin{split}
\langle \Ff h, \Ff \varphi\rangle &= \langle h,\varphi\rangle\\
&= 8\pi\i\int_{U'}(\Ff q)(\omega,\sigma)\overline{\omega(\Ff\varphi)(\sgn\omega\sqrt{\omega^2-\sigma^2}, \sigma)}\d\omega\d\sigma\\
&= 8\pi\i\int_{U}\varrho(\Ff q)(\sgn\varrho\sqrt{\varrho^2+|\sigma|^2},\sigma)\overline{(\Ff\varphi)(\varrho, \sigma)}\d\varrho\d\sigma,
\end{split}
\end{equation*}
which finishes the proof.
\end{proof}

For any  $f\in\S'(\R^+ \times \R^2)$, its Fourier-Laplace transform $(\Ff f)(z, \sigma)$ is holomorphic with respect to $z$ in the upper half plane $\H^+$. Define
\begin{equation}\label{eq:Skappainverse}
\S'_{\kappa^{-1}}(\R^+ \times \R^2)= \{f\in \S'(\R^+ \times \R^2): \forall z\in\H^+, (\Ff f)(z, \sigma) \text{ can be analytically continued to } z\in \kappa^{-1}(\H^+)\}.
\end{equation}

\begin{lemma}
For every $f(t, \vxi)\in \S'(\R^+ \times \R^2)$, $Af(t, \vxi)\in \S'_{\kappa^{-1}}(\R^+ \times \R^2)$.
\end{lemma}
\begin{proof}
For any $f(t, \vxi)\in\S'(\R^+ \times \R^2)$, let $g(t, \vxi)=Af(t, \vxi)$, combining this with \autoref{Operator-A} deduces $(\Ff g)(z, \sigma)=(\Ff f)(\kappa(z), \sigma)$ for all $z\in \H^+$.
Define a holomorphic function
\begin{equation*}
g_0(z, \sigma)=(\Ff f)(\kappa(z), \sigma), \; \text{ for all } z\in\kappa^{-1}(\H^+).
\end{equation*}

Directly $\Ff g$ and $g_0$ coincides in $z\in\H^+$. Consequently, $g_0(z, \sigma)$ is the analytic continuation of $(\Ff g)(z, \sigma)$ to the region $\kappa^{-1}(\H^+)$. This shows that $g\in \S'_{\kappa^{-1}}(\R^+ \times \R^2)$.
\end{proof}

\begin{theorem}\label{Thm-reconstructionA}
Suppose that $q^a\in\S'_{\kappa^{-1}}(\R^+\times\R^2)$, so that $\Ff q^a(z, \cdot)$ is a distribution in $\S'(\R^2)$ for all $z\in \kappa^{-1}(\H^+)$. Let $U\subset\R\times\R^2$ be an open region, and $U_\omega=\{\sigma\in\R^2\mid(\omega,\sigma)\in U\}$ be the "slices" of $U$ for every $\omega\in\R$\}. If for every $\omega\in\R$ such that $U_\omega\ne\emptyset$, the distribution $\lim_{z\to\kappa^{-1}(\omega)}\Ff q^a(z,\cdot)\in\S'(\R^2)$ exists and is regular on the set $U_\omega=\{\sigma\in\R^2\mid(\omega,\sigma)\in U\}$.

Then $\Ff q(\omega,\sigma)$ is regular on $U$, and the following equality holds for all $(\omega,\sigma)\in U$,

\begin{equation*}
(\Ff q)(\omega,\sigma)=(\Ff q^a)(\kappa^{-1}(\omega),\sigma).
\end{equation*}
\end{theorem}
\begin{proof}
Since for any $\varphi(t, \xi)\in \S(\R^+\times\R^2)$, the operator $A$ satisfies
\begin{equation*}
\begin{split}
\langle \Ff(Aq)( \omega, \sigma), \Ff \varphi(\omega, \sigma)\rangle &= \langle \Ff q^a(\omega, \sigma), \Ff \varphi(\omega, \sigma)\rangle\\
&=\langle \Ff q(\kappa(\omega), \sigma), \Ff \varphi(\omega, \sigma)\rangle,
\end{split}
\end{equation*}
and from assumption that $\lim_{z\to\kappa^{-1}(\omega)}\Ff q^a(z,\cdot)\in\S'(\R^2)$ exists, it implies that the distribution $\lim_{z\to\omega}\Ff q(z,\cdot)$ exists and is regular in $U_\omega$. Therefore, for all $\varphi$ such that $\Ff\varphi$ is supported in $U$,
\begin{equation*}
\begin{split}
\langle \Ff q, \Ff \varphi\rangle&= \langle q, \varphi\rangle\\
&= \lim_{\eta\to0}\int_{\R}\langle (\Ff q)(\omega+\eta i,\cdot),(\Ff\varphi)(\omega,\cdot)\rangle\d\omega\\
&= \int_{U_\omega\ne\emptyset}\lim_{\eta\to0}\left(\int_{U_\omega}(\Ff q^a)(\kappa^{-1}(\omega+\eta i),\sigma)\overline{(\Ff\varphi)(\omega,\sigma)}\d\sigma\right)\d\omega\\
&= \int_{U}(\Ff q^a)(\kappa^{-1}(\omega),\sigma)\overline{(\Ff\varphi)(\omega,\sigma)}\d\sigma\d\omega.
\end{split}
\end{equation*}
Thus, $(\Ff q)(\omega,\sigma)$ is equal to the function $(\Ff q^a)(\kappa^{-1}(\omega),\sigma)$ in the region $U$.
\end{proof}

\begin{corollary}
Assume the conditions of \autoref{Thm-reconstructionR} and \autoref{Thm-reconstructionA} hold, then $h(\vx)$ is reconstructed by
\begin{equation*}
(\Ff h)(\sigma,\varrho)=8\pi\i\varrho(\Ff q^a)\left(\kappa^{-1}(\sgn\varrho\sqrt{\varrho^2+|\sigma|^2}),\sigma)\right).
\end{equation*}
\end{corollary}

\subsection{Existence and Uniqueness}

In this part we will describe that the range of the operator $R$ is exactly the space $\S'_{\text{cone}}$. After this on one hand we could prove $R$ is a bijection mapping from $\S'(\R^+\times\R^2)$ to $\S'_{\text{cone}}$. On the other hand the operator $A$ is also a bijection, so the composition $R^a=AR$ is also bijective. This gives us the existence and uniqueness of source term for the inverse problem.


\begin{lemma}
For every $f\in S'$, $Rf\in S'_{\text{cone}}$.
\end{lemma}
\begin{proof}
Let $\varphi$ be an arbitrary function in $S$ such that $(\Ff\varphi)(\omega, \sigma)=0$ for any $|\omega|\geq |\sigma|$. Then directly,
\begin{equation}\label{Rrangeadjoint}
\langle Rf, \varphi\rangle=\langle f, R^*\varphi\rangle=\langle\Ff f, \Ff (R^*\varphi)\rangle.
\end{equation}

Recall the representation of $R^*$ in \autoref{Rstar-frequency}. It is obvious that $\sqrt{|\sigma|^2+\varrho^2} \geq |\sigma|$. Using the assumptions on $\varphi$, we infer that $\Ff\varphi(\sqrt{|\sigma|^2+\varrho^2}, \sigma)=\Ff\varphi(-\sqrt{|\sigma|^2+\varrho^2}, \sigma)=0$, which deduces
\begin{equation*}
(\Ff (R^*\varphi))(\sigma, \varrho)=0
\end{equation*}
for all $(\sigma, \varrho)\in \R^2 \times \R^+$, plugging this into \autoref{Rrangeadjoint} yields $\langle Rf, \varphi \rangle=0$. This proves $Rf\in \S'_{\text{cone}}$.
\end{proof}

\begin{lemma}\label{Thm-R-surjective}
For every $g\in\S'_{cone}(\R^+\times\R^2)$, there exists a distribution $f\in\mathcal{S}'(\R^+\times\R^2)$
such that $g=Rf$, which deduces $R$ is surjective.
\end{lemma}
\begin{proof}
Let $f\in\S'(\R^+\times\R^2)$ be the distribution defined by
\begin{equation*}
  \langle f,\varphi\rangle=\langle g, E\varphi\rangle
\end{equation*}
for all $\varphi\in\S(\R^+\times\R^2)$. Using \autoref{Relation-ER},
\begin{equation*}
\begin{split}
\langle Rf,\varphi\rangle&=\langle f,R^*\varphi\rangle\\
&=\langle g,ER^*\varphi\rangle\\
&=\langle g,\varphi\rangle.
\end{split}
\end{equation*}

so $g=Rf$. This means $R$ is surjective.
\end{proof}

\begin{lemma}\label{Thm-R-injective}
Suppose that $h\in\S'(\R^+\times\R^2)$ such that $Rh=0$, then $h=0$, which means $R$ is injective.
\end{lemma}
\begin{proof}
Suppose that $Rh=0$, for any $\varphi\in \S_e$, by \autoref{Thm:RstarE} we have,
\begin{equation*}
\langle h, \varphi\rangle = \langle h, R^*E \varphi\rangle = \langle Rh, E\varphi\rangle=0,
\end{equation*}
which deduces $h=0$.

\end{proof}

Combining \autoref{Thm-R-surjective} and \autoref{Thm-R-injective} shows that $R$ is a bijection from $\S'(\R^+\times\R^2)$ to $\S'_{\text{cone}}$. The next step is to prove that $A$ is also a bijection from $\S'(\R^+\times\R^2)$ to $\S'_{\kappa^{-1}}(\R^+\times\R^2)$, which finally concludes the bijectivity of $R^a$.

\begin{lemma}\label{Thm-A-surjective}
For every $q^a\in \S'_{\kappa^{-1}}(\R^+\times\R^2)$, there exists a distribution $q\in \S'(\R^+\times\R^2)$ such that $q^a=Aq$.
\end{lemma}

\begin{proof}
By the definition of the space $\S'_{\kappa^{-1}}(\R^+\times\R^2)$, the function $\Ff q^a(z, \sigma)$ is holomorphic in the region $\kappa^{-1}(\H^+)\times\R^2$.

Since $\kappa^{-1}(z)$ is holomorphic from \autoref{deAttCoeff} (ii), define a holomorphic function
\begin{equation*}
q_0(z, \sigma)=\Ff q^a(\kappa^{-1}(z), \sigma),\; \text{ for all } z\in\H^+.
\end{equation*}
 By \autoref{Hormander}, $q_0$ is the Fourier-Laplace transform of a distribution $q\in\S'(\R^+\times\R^2)$.

Recall the representation of the operator $A$ in \autoref{Operator-A},
\begin{equation*}
\Ff Aq(z, \sigma)=\Ff q(\kappa(z), \sigma)=q_0(\kappa(z), \sigma)=\Ff q^a(z, \sigma)
\end{equation*}
for all $z\in \H^+$. This means $Aq=q^a$ as claimed.
\end{proof}

\begin{lemma}\label{Thm-A-injective}
Suppose that $q\in\S'(\R^+\times\R^2)$ such that $Aq=0$, then $q=0$, which means operator $A$ is injective.
\end{lemma}
\begin{proof}
Suppose that $q\in\S'(\R^+\times\R^2)$ such that $Aq=0$. Then using \autoref{Operator-A},
\begin{equation*}
\Ff q(\kappa(\omega), \sigma)=\Ff Aq(\omega, \sigma)=0
\end{equation*}
holds for every $\omega\in\R$. This means the holomorphic function $z\mapsto \Ff q(z, \sigma)$ vanishes on the curve $\kappa(\R)$, and therefore it must equal to zero everywhere in $\H^+$. Consequently, $\Ff q=0$, which implies $q=0$.
\end{proof}

Combining \autoref{Thm-A-surjective} and \autoref{Thm-A-injective} shows that $A$ is a bijection from $\S'(\R^+\times\R^2)$ to $\S'_{\kappa^{-1}}(\R^+\times\R^2)$. It follows that
\begin{corollary}
$h\in S'(\R^2\times\R^+)$, and $q^a(t, \vxi)\in\S'_{\kappa^{-1}}(\R^2\times\R^+)$ be the integrated measurement.  The operator $R^a$ is a bijection, which means the reconstruction formula \autoref{eq:reconstructionPlane} is unique.
\end{corollary}

\section{For a sphere observation surface}\label{sec:sphere}

In this section, we will demonstrate the viability of the reconstruction formula for a sphere observation surface.

\subsection{Filtered backprojection}

In this part, we consider the observation surface $\partial\Omega$ is a sphere $\partial B(0,\varpi)$.

The favourite way of inverting operator $R$ is the filtered backprojection formulas. Several different variations of such formulas were developed and various versions of the 3D inversion formulas could be found in \cite{KucKun08, Kuc12}. The classical backprojection formula is given by
\begin{equation}\label{eqBackProjection-withoutAttenuation}
h(\vx)=-\frac{1}{8\pi^2 \varpi}\Delta_{\vx}\int_{\partial\Omega}\frac{q(|\vxi-\vx|, \vxi)}{|\vxi-\vx|}\d S(\vxi),
\end{equation}
where $q(t, \vxi)$ is defined in \autoref{Integrated-p}.

However, since there is attenuation in the media, the measurement is actually $q^a(t,\vxi)$ instead of $q(t,\vxi)$. If the formula \autoref{eqBackProjection-withoutAttenuation} is directly used on the real data $q^a(t,\vxi)$, the reconstructed image will not be the real $h(\vx)$.

\subsection{reconstruction formula}

Recall in the frequency domain the solution of standard wave equation and attenuated wave equation in \autoref{freqsol-noatt} and \autoref{freqsol-att}, doing the inverse Fourier transform for them with respect to $\omega$, we get the expression of $q$ and $q^a$ in the time domain,
\begin{equation}\label{q-texpress}
q(t, \vxi)= \int_{\R^3} \frac{\delta(t-|\vxi-\vx|)}{t}h(\vx)\d \vx,
\end{equation}
and
\begin{equation}\label{qa-texpress}
q^a(t, \vxi)= \int_{\R^3} \frac{e^{-\kappa_{\infty}|\vxi-\vx|}}{|\vxi-\vx|}\Ff^{-1}_1(e^{\i\omega|\vxi-\vx|}e^{\kappa_*|\vxi-\vx|})h(\vx)\d \vx.
\end{equation}
where $\Ff^{-1}_1$ denotes the inverse Fourier transform with respect to the first variable, which in this section is $\omega$. Inserting $\kappa(\omega) = \frac\omega c+\i\kappa_\infty+\kappa_*(\omega)$ into \autoref{qa-texpress}, we get
\begin{equation}\label{qa-texpress1}
q^a(t, \vxi)= \int_{\R^3} \frac{e^{-\kappa_{\infty}|\vxi-\vx|}}{|\vxi-\vx|}\Ff^{-1}_1(e^{\i\kappa_*(\omega)|\vxi-\vx|})(t-|\vxi-\vx|)h(\vx)\d \vx.
\end{equation}

Since $\kappa_*$ is bounded from \autoref{deWeakAttenuation}, the Taylor expansion
\begin{equation*}
e^{\i\kappa_*(\omega)|\vxi-\vx|}=\sum_{j=0}^{\infty}\frac{(\i\kappa_*(\omega)|\vxi-\vx|)^j}{j!}
\end{equation*}
with respect to $\kappa_*(\omega)$ is uniformly convergent. Since $\Ff^{-1}_1(e^{\i\kappa_*(\omega)|\vxi-\vx|})$ is the inverse Fourier transform only with respect to $\omega$, so it is equal to $\sum_{j=0}^{\infty}\frac{|\vxi-\vx|^j}{j!}\Ff_1^{-1} ((\i\kappa_*(\omega))^j)$. Plugging this into \autoref{qa-texpress1}, and then
\begin{equation}\label{qa-texpress2}
q^a(t, \vxi)=\int_{\R^3}\frac{e^{-\kappa_{\infty}|\vxi-\vx|}}{|\vxi-\vx|}\sum_{j=0}^{\infty}\frac{|\vxi-\vx|^j}{j!}\Ff^{-1}_1((\i\kappa_*(\omega))^j)(t-|\vxi-\vx|)h(\vx)\d \vx.
\end{equation}


If we denote
\begin{equation}\label{Func-rj}
r_j(t-|\vxi-\vx|)= (\Ff^{-1}((\i\kappa_*(\omega))^j))(t-|\vxi-\vx|),
\end{equation}
it is easy to see $r_j(t-|\vxi-\vx|)=\int_{\R}\delta(\tau-(t-|\vxi-\vx|))r_j(\tau)\d\tau$, and then plugging this formula and \autoref{Func-rj} into \autoref{qa-texpress2} yields
\begin{align*}
q^a(t, \vxi) &= \int_{\R^3}e^{-\kappa_{\infty}|\vxi-\vx|}\sum_{j=0}^{\infty}\frac{|\vxi-\vx|^{j-1}}{j!}r_j(t-|\vxi-\vx|)h(\vx)\d \vx\\
&= \sum_{j=0}^{\infty}\frac{1}{j!}\int_{\R^3}e^{-\kappa_{\infty}|\vxi-\vx|}|\vxi-\vx|^{j-1}\int_{\R}\delta(\tau-(t-|\vxi-\vx|))r_j(\tau)\d\tau h(\vx)\d \vx.
\end{align*}

Noticing that $\delta$ is supported on $|\vxi-\vx|=t-\tau$, so
\begin{equation*}
q^a(t, \vxi)= \sum_{j=0}^{\infty}\frac{1}{j!}\int_{\R^3}e^{-\kappa_{\infty}(t-\tau)}(t-\tau)^{j-1}\int_{\R}\delta(\tau-(t-|\vxi-\vx|))r_j(\tau)
\d\tau h(\vx)\d \vx.
\end{equation*}

Apply the multiplication operator $M=e^{\kappa_{\infty}t}$ on $q^a(t, \vxi)$, we have
\begin{equation*}
\begin{split}
e^{\kappa_{\infty}t}q^a(t, \vxi)&= \sum_{j=0}^{\infty}\frac{1}{j!}\int_{\R}(t-\tau)^{j-1}r_j(\tau)e^{\kappa_{\infty}\tau}\int_{\R^3}\delta(\tau-(t-|\vxi-\vx|)) h(\vx)\d \vx \d\tau\\
&=\sum_{j=0}^{\infty}\frac{1}{j!}\int_{\R}(t-\tau)^{j}r_j(\tau)e^{\kappa_{\infty}\tau}q(t-\tau, \vxi) \d\tau,
\end{split}
\end{equation*}
where the last equality holds by noticing the expression of $q(t, \vxi)$ in \autoref{q-texpress}.

We note that the summand in the above equation when $j=0$ is equal to
\begin{equation*}
\int_{\R}r_0(\tau)e^{\kappa_{\infty}\tau}q(t-\tau, \vxi) \d\tau = \int_{\R}\delta(\tau)e^{\kappa_{\infty}\tau}q(t-\tau, \vxi) \d\tau = q(t, \vxi).
\end{equation*}

Therefore, we could get the expression of $e^{\kappa_{\infty}t}q^a(t, \vxi)-q(t, \vxi)$,
\begin{equation}\label{qa-q-relation}
e^{\kappa_{\infty}t}q^a(t, \vxi)-q(t, \vxi)=\sum_{j=1}^{\infty}\frac{1}{j!}\int_{\R}(t-\tau)^{j}r_j(\tau)e^{\kappa_{\infty}\tau}q(t-\tau, \vxi) \d\tau.
\end{equation}

Denote
\begin{equation}\label{eqBackProjection-attenuation}
h^a(\vx)=-\frac{1}{8\pi^2 \varpi}\Delta_{\vx}\int_{\partial\Omega}e^{\kappa_{\infty}|\vxi-\vx|}\frac{q^a(|\vxi-\vx|, \vxi)}{|\vxi-\vx|}\d S(\vxi).
\end{equation}

From \autoref{eqBackProjection-attenuation} and \autoref{eqBackProjection-withoutAttenuation}, we know
\begin{equation}\label{ha-h-original}
h^a(\vx)-h(\vx)=-\frac{1}{8\pi^2\varpi}\Delta_\vx\int_{\partial\Omega}(\frac{e^{\kappa_{\infty}|\vxi-\vx|}}{|\vxi-\vx|}q^a(|\vxi-\vx|, \vxi)-q(|\vxi-\vx|, \vxi))\d S(\vxi).
\end{equation}
By plugging\autoref{qa-q-relation} into \autoref{ha-h-original}, we could compute the relation between $h^a(\vx)$ and $h(\vx)$.
\begin{align*}
h^a(\vx)-h(\vx) = -\frac{1}{8\pi^2\varpi}\Delta_{\vx}\int_{\partial\Omega}\sum_{j=1}^{\infty}\frac{1}{j!}\int_{\R}
(|\vxi-\vx|-\tau)^{j-1}r_j(\tau)e^{\kappa_\infty\tau}q(|\vxi-\vx|-\tau, \vxi)\d \tau\d S(\vxi).
\end{align*}

Inserting \autoref{q-texpress} into above yields
\begin{align*}
h^a(\vx)-h(\vx) &=-\frac{1}{8\pi^2\varpi}\Delta_{\vx}\int_{\partial\Omega}\int_{\Omega}\sum_{j=1}^{\infty}\frac{1}{j!}\int_{\R}
(|\vxi-\vx|-\tau)^{j-1}r_j(\tau)e^{\kappa_\infty\tau}h(\vy)\delta(|\vxi-\vy|-|\vxi-\vx|+\tau)\d \vy\d \tau\d S(\vxi)\\
&=-\frac{1}{8\pi^2\varpi}\Delta_{\vx}\int_{\partial\Omega}\int_{\Omega}\sum_{j=1}^{\infty}\frac{1}{j!}|\vxi-\vy|^{j-1}r_j(|\vxi-\vx|-|\vxi-\vy|)e^{\kappa_\infty(|\vxi-\vx|-|\vxi-\vy|)}h(\vy)\d \vy\d S(\vxi).
\end{align*}

Therefore, we can write $h^a(\vx)-h(\vx)$ as $Th(\vx)$, where $T$ is an integral operator with kernel
\begin{equation}\label{kernelT}
F_T(\vx,\vy)=-\frac{1}{8\pi^2\varpi}\sum_{j=1}^{\infty}\frac{1}{j!}\Delta_\vx\int_{\partial\Omega}|\vxi-\vy|^{j-1}r_j(|\vxi-\vx|-|\vxi-\vy|)e^{\kappa_\infty(|\vxi-\vx|-|\vxi-\vy|)}\d S(\vxi).
\end{equation}

Finally, let us sum up our reconstruction formula for the sphere observation surface.
\begin{equation}
h=(I+T)^{-1}B_pMq^{a},
\end{equation}
where\begin{itemize}
       \item $M$ is a multiplication operator satisfying $Mq^{a}=e^{\kappa_\infty t}q^{a}(t, \vxi)$;
       \item $B_p$ is the Filtered backprojection formula defined in \autoref{eqBackProjection-attenuation};
       \item $T$ is an integral operator with kernel given by \autoref{kernelT}.
     \end{itemize}

We will prove later that under certain additional assumptions on $\kappa_*(\omega)$, the operator $T$ will be compact, and therefore the singular values of $(I+T)$ have a positive lower bound. Thus $(I+T)^{-1}$ could be calculated by a matrix inversion numerically without ill-conditioning. Also, in the next subsection we will prove this decomposition of the operator $(R^a)^{-1}$ is feasible.

\subsection{Space clarification}

Assume $h(\vx)\in L^2(\Omega_\epsilon)$, to clarify the range and domain of the operators $(I+T)^{-1}$, $B_p$ and $M$ could coincide, we will introduce their domain and range respectively.

\begin{lemma}[\cite{Pal10}]\label{thmPalamodov}
Suppose that the observation surface $\partial\Omega$ is a smooth convex surface. Then the operator $R$ defined in \autoref{R} is a bounded linear operator from $L^2(\Omega_\epsilon)$ to $H^1(\R^+\times\partial\Omega)$, and the adjoint operator $R^*$ is a bounded linear operator from $H^1(\R^+\times\partial\Omega)$ to $H^2(\Omega_\epsilon)$.
\end{lemma}

\begin{theorem}
For every $h\in L^2(\Omega_\epsilon)$, let $M$ be the multiplication operator satisfying $Mq^{a}(t, \vxi)=e^{\kappa_\infty t}q^{a}(t, \vxi)$, then $Mq^a\in H^1(\R^+\times\partial\Omega)$.
\end{theorem}

\begin{proof}
\autoref{thmPalamodov} deduces $q(t, \vxi)\in \H^1(\R^+\times\partial\Omega)$. From the assumption that $\kappa_*\in L^2$ in \autoref{deWeakAttenuation}, the representation of $r_j(\tau)$ in \autoref{Func-rj} belongs to $L^2(\R)$. \autoref{qa-q-relation} deduces that $Mq^a-q$ is the result of an integral operator with $L^2$ kernel acting on $q$, and therefore belongs to $L^2(\R^+\times\partial\Omega)$. Similarly, by differentiating \autoref{qa-q-relation} with respect to $t$, $\frac{\partial}{\partial t}(Mq^a-q)$ is the result of an integral operator with $L^2$ kernel acting on $\frac{\partial}{\partial t}q$, so it also belongs to $L^2(\R^+\times\partial\Omega)$. It follows that $Mq^a-q\in \H^1(\R^+\times\partial\Omega)$, and therefore $Mq^a\in H^1(\R^+\times\partial\Omega)$.
\end{proof}

Again using \autoref{thmPalamodov}, the adjoint operator $R^*$ is a bounded linear operator from $H^1(\R^+\times\partial\Omega)$ to $H^2(\Omega_\epsilon)$. Recall the expression of the operator $R^*$ in \autoref{Rstar_time}, the relationship $$(B_pq)(\vx)=C\Delta_{\vx}(R^*q)(\vx)$$ can be obtained. This means the filtered backprojection $B_p$ is a bounded operator from $\H^1(\R^+\times\partial\Omega)$ to $L^2(\Omega_\epsilon)$.

To identify the domain and range of operator $T$, noticing $T$ is an integral operator with kernel \autoref{kernelT}, we first give the boundedness of $r_j$.

\begin{lemma}\label{Thm-rj-bounded}
Let $r_j(t)$ be defined in \autoref{Func-rj} for every $j\in Z^+$, if $(\Ff r_j)(\omega)$ is bounded and there exist some constants $b_m$ such that $(\Ff r_j)(\omega)$ has an asymptotic expansion
\begin{equation}\label{Expansion-Fourier-rj}
(\Ff r_j)(\omega)\sim\sum_{m=1}^{\infty}b_m\omega^{-m}
\end{equation}
at $\omega\to\infty$, then $r_j(t)$ is bounded in $\R^+$.
\end{lemma}
\begin{proof}
Using the asymptotic expansion \autoref{Expansion-Fourier-rj}, when $\omega\to\infty$,
\begin{equation*}
\begin{split}
(\Ff r_j)(\omega)-\frac{b_1}{\omega+i}&= \sum_{m=1}^{\infty}b_m\omega^{-m}-\frac{b_1}{\omega+i}\\
&= \frac{b_1}{\omega}-\frac{b_1}{\omega+i}+\sum_{m=2}^{\infty}b_m\omega^{-m}\\
&= \frac{\i b_1}{\omega(\omega+i)}+\sum_{m=2}^{\infty}b_m\omega^{-m}.
\end{split}
\end{equation*}

Since $\frac{\i b_1}{\omega(\omega+i)}+\sum_{m=2}^{\infty}b_m\omega^{-m}=O(|\omega|^{-2})$ when $|\omega|\to\infty$, we conclude that $(\Ff r_j)(\omega)-\frac{b_1}{\omega+i}$ is an $L^1$ function. It is obvious that the inverse Fourier transform of an $L^1$ function is an $L^\infty$ function. Also, the Inverse Fourier transform of $\frac{b_1}{\omega+i}$ is $-b_1 i\sqrt{2\pi} H(t)e^{-t}$ where $H$ is the Heaviside function, as proved below.
\begin{equation*}
\begin{split}
\frac{1}{\sqrt{2\pi}}\int_{\R}e^{i\omega t}(-ib_1\sqrt{2\pi}H(t)e^{-t})\d t&=-b_1 i\int_{0}^{\infty}e^{(i\omega-1) t}\d t\\
&=-b_1i\left.\frac{e^{(i\omega-1) t}}{i\omega-1}\right|^\infty_0\\
&=\frac{b_1i}{i\omega-1}\\
&=\frac{b_1}{\omega+i}.
\end{split}
\end{equation*}

Therefore,  $r_j(t)$ is equal to the sum of $-b_1 i\sqrt{2\pi} H(t)e^{-t}$ and an $L^\infty$ function, so it is bounded.
\end{proof}

\begin{theorem}
If $\kappa_*(\omega)$ has an asymptotic expansion at $\omega\to\infty$
\begin{equation}\label{eq:kappastarassumption2}
\kappa_*(\omega)\sim\sum_{m=1}^{\infty}\kappa_m\omega^{-m}
\end{equation}
Then the integral operator $T$ with kernel defined by \autoref{kernelT} is a bounded operator from $L^2(\Omega_\varepsilon)$ to $L^2(\Omega_\varepsilon)$.
\end{theorem}
\begin{proof}
Let
\begin{equation}\label{eq:F0}
F_0(\vx,\vy)=-\frac{1}{8\pi^2\varpi}\sum_{j=1}^{\infty}\frac{1}{j!}\int_{\partial\Omega}|\vxi-\vy|^{j}r_j(|\vxi-\vx|-|\vxi-\vy|)e^{\kappa_\infty(|\vxi-\vx|-|\vxi-\vy|)}\d S(\vxi),
\end{equation}
so that $F_T(\vx,\vy)=\Delta_xF_0(\vx,\vy)$.

The aim is to prove that $|F_T(\vx,\vy)|$ can be bounded above by $C|\vx-\vy|^{-2}$ for some constant $C$. To do this, we will prove that when $\vx\to \vy$, the kernel $F_0(\vx,\vy)$ has an asymptotic expansion
\begin{equation}\label{eq:F0AsymptoticExpansion}
F_0 (\vx,\vy)\sim\sum_{m=0}^{\infty} f_m(\vy,\frac{\vx-\vy}{|\vx-\vy|})|\vx-\vy|^m
\end{equation}
where $f_m$ are smooth functions on $\Omega_\varepsilon\times S^2$. Here $S^2\subset \R^3$ is the unit sphere. We will focus on the $j=1$ term in \autoref{eq:F0}. The other terms with $j\geq2$ will be processed by essentially the same argument.

Using \autoref{deAttCoeff}(ii) and \autoref{Hormander} originally defined in \cite{Hoe03}, the functions $r_j(t)=0$ when $t<0$. Therefore, we will investigate the properties of $r_j(t)$ when $t>0$.

First, we will prove that all derivatives of the function $r_1(t)=\Ff^{-1}(i\kappa_*(\omega))$ are bounded in $\R^+$. By using the asymptotic expansion \autoref{eq:kappastarassumption2},
\begin{equation*}
\begin{split}
(\Ff r_1^{(l)})(\omega)&=i(-i\omega)^l\kappa_*(\omega)\\
&\sim i(-i)^l\sum_{m=0}^{\infty}\kappa_m\omega^{l-m}\\
&=i(-i)^l\left(\sum_{m=0}^{l}\kappa_m\omega^{l-m}+\sum_{m=l+1}^{\infty}\kappa_m\omega^{l-m}\right).
\end{split}
\end{equation*}
Noticing that $\sum_{m=0}^{l}\kappa_m\omega^{l-m}$ is a polynomial in $\omega$, and $\sum_{m=l+1}^{\infty}\kappa_m\omega^{l-m}$ is a function satisfying the assumptions of \autoref{Thm-rj-bounded}. Therefore, $\Ff^{-1}\left(\sum_{m=0}^{l}\kappa_m\omega^{l-m}\right)$ is a distribution supported at $\{0\}$, and $\Ff^{-1}\left(\sum_{m=l+1}^{\infty}\kappa_m\omega^{l-m}\right)$ is a bounded function supported in $\R^+$. Thus every derivative of $r_1$ are bounded in $\R^+$, as claimed.

For the terms $r_j^{(l)}(t)$ with $j\geq2$, by using the asymptotic expansion \autoref{eq:kappastarassumption2},
\begin{equation*}
\begin{split}
(\Ff r_j)(\omega)&=(i\kappa_*(\omega))^j\\
&\sim i^j\left(\sum_{m=1}^{\infty}\kappa_m\omega^{-m}\right)^j\\
&\sim i^j\sum_{m=j}^{\infty}P_{j,m}(\kappa_1,\kappa_2,\dots)\omega^{-m},
\end{split}
\end{equation*}
where $P_{j,m}$ are polynomials of degree $j$. Therefore, for any $l\geq 0$, $r_j^{(l)}(t)$ is bounded in $\R^+$.

Let $r_j^*(t)=e^{\kappa_{\infty}t}r_j(t)$. Based on the boundedness of $r_j^{(l)}(t)$, $r_j(t)$ is amooth in $\R^+$, so $r^*_j(t)$ is also smooth in $\R^+$ because $e^{\kappa_\infty t}$ is smooth. Thus the following asymptotic expansion holds when $t\to 0^+$,
\begin{equation*}
r_j^*(t)\sim\sum_{m=0}^{\infty}d_{j,m}t^m
\end{equation*}
for all $j\geq 1$, where $d_{j,m}\in\C$. For any fixed $m\geq0$, the numbers $|d_{j,m}|$ can be bounded above by $(C_m)^j$ for some constants $C_m$.

Based on the smoothness of the observation surface $\partial\Omega$ and the smoothness of the functions $r_j^*(t)$ in $\R^+$, from \autoref{eq:F0}, $F_0(\vx,\vy)$ is smooth outside of the main diagonal $\vx=\vy$. We will now prove the asymptotic expansion \autoref{eq:F0AsymptoticExpansion} holds.

Let $\vx=\vy+\rho\eta$ with $\rho>0$ and $\eta\in S^2$. For any $\xi\in\partial\Omega$, the difference $|\vxi-\vy-\rho\eta|-|\vxi-\vy|$ is a smooth function of $\rho$, and can be written as $\frac{-\langle\eta,\vxi-\vy\rangle}{|\vxi-\vy|}\rho+\frac{|\vxi-\vy|^2-\langle\eta,\vxi-\vy\rangle^2}{2|\vxi-\vy|^3}\rho^2+O(\rho^3)$. The coefficients for the linear term and the quadratic term cannot both be zero, so one of these two terms dominates when $\rho\to0^+$. Therefore, when $\rho\to 0^+$, $|\vxi-\vy-\rho\eta|-|\vxi-\vy|$ goes to either $0^+$ or $0^-$.

Using the smoothness and expansion of the function $r_j^*$, we can conclude that the function $|\vxi-\vy|^{j}r_j^*(|\vxi-\vy-\rho\eta|-|\vxi-\vy|)$ has an asymptotic expansion with respect to $\rho$ when $\rho\to0^+$.

Since all the functions involved are bounded above by $C^j$ for some $C$, we can conclude the series
\begin{equation*}
\sum_{j=1}^{\infty}\frac{1}{j!}|\vxi-\vy|^{j}r_j^*(|\vxi-\vy-\rho\eta|-|\vxi-\vy|)
\end{equation*}
converges uniformly, and therefore can also be written as $\sum_{m\geq 0}f_m^*(\vy,\vxi,\eta)\rho^m$ for some smooth functions $f_m^*$.
\begin{equation}\label{eq:F0Expansion}
\begin{split}
F_0(\vy+\rho\eta,\vy)&=C\int_{\partial\Omega}\sum_{j=1}^{\infty}\frac{1}{j!}|\vxi-\vy|^{j}r_j^*(|\vxi-\vy-\rho\eta|-|\vxi-\vy|)\d S(\vxi)\\
&=C\int_{\partial\Omega}\sum_{m\geq 0}f_m^*(\vy,\xi,\eta)\rho^m\d S(\vxi)\\
&=C\sum_{m\geq 0}f_m(\vy,\eta)\rho^m
\end{split}
\end{equation}
where $f_m(\vy,\eta)=\int_{\partial\Omega}f_m^*(\vy,\xi,\eta)\d S(\vxi)$.

Having established the expansion \autoref{eq:F0Expansion}, we can apply the Laplace operator with respect to $\vx$ on it to get
\begin{equation*}
\begin{split}
F_T(\vx,\vy)&=\Delta_\vx F_0(\vx,\vy)\\
&=\sum_{m=0}^{\infty} \Delta_\vx f_m(\vy,\frac{\vx-\vy}{|\vx-\vy|})|\vx-\vy|^m\\
&=\sum_{m=0}^{\infty} \left(\frac{\partial^2}{\partial \rho^2}+\frac{2}{\rho}\frac{\partial}{\partial \rho}+\frac{1}{\rho^2}\Delta_\eta\right) f_m(\vy,\eta)\rho^m\\
&=\sum_{m=0}^{\infty} (m(m+1)f_m(\vy,\eta)+\Delta_{\eta}f_m(\vy,\eta))\rho^{m-2},
\end{split}
\end{equation*}
which means that $|F_T(\vx,\vy)|\leq C|\vx-\vy|^{-2}$, so $T$ is an integral operator with a weakly singular kernel. Using Theorem 1 of \cite{Kos74}, $T$ is a compact operator from $L^2(\Omega_\varepsilon)$ to $L^2(\Omega_\varepsilon)$.

\end{proof}

\begin{corollary}
Let $T$ be the integral operator with kernel \autoref{kernelT} and $\kappa_*(\omega)$ has an asymptotic expansion \autoref{eq:kappastarassumption2}, then $(I+T)^{-1}$ is a bounded operator from $L^2(\Omega_\varepsilon)$ to $L^2(\Omega_\varepsilon)$.
\end{corollary}

\section*{Appendix}

Suppose that $f\in \mathcal{S}'(R^+)$ is a tempered distribution. We consider the function $e_z(t)=e^{izt}$ for $t>0$. If $\Im z>0$, the function is exponentially decreasing when $t\to+\infty$, so we can extend it to a Schwartz function $\tilde{e}_z(t)$ on $\R$. The Fourier-Laplace transform of $f$ is defined as
\begin{equation}
\Ff f(z)=\frac{1}{\sqrt{2\pi}}\langle f,\tilde{e}_z\rangle.
\end{equation}
Since $f$ is supported on $\R^+$, this inner product does not depend on the extension.

The following theorem shows that the function $\Ff f(z)$ is holomorphic on the upper half-plane $\H^+$ (for example, by verifying the Cauchy-Riemann equations). Conversely, every holomorphic function on $\H^+$ satisfying a growth condition is the Fourier-Laplace transform of a distribution supported in $\R^+$.

\begin{theorem}[\cite{Hoe03}]\label{Hormander}
Let $f\in \mathcal{S}'(\R^+)$ be a tempered distribution. Then its Fourier-Laplace transform $\Ff f(z)$ is holomorphic for $z\in\H^+$. Furthermore, for every $\eta>0$, there exists constants $C>0$ and $N\in \N$ such that
\begin{equation}
|\Ff f(z)|\leq C(1+|z|)^N
\end{equation}
for every $z\in \H^+$ with $\Im z\geq\eta$.

Conversely, If there is a holomorphic function $u:\H\to\C$, and there exist constants constants $C>0$, $\eta>0$ and $N\in \N$ such that
\begin{equation}
|u(z)|\leq C(1+|z|)^N
\end{equation}
for every $z\in \H^+$ with $\Im z\geq\eta$, Then there exists a tempered distribution $f$ supported on $\R^+$ such that $u=\Ff f$.
\end{theorem}

\newpage
\subsection*{Acknowledgements}
This research was funded in whole, or in part, by the Austrian Science Fund 10.55776/T1160 ``Photoacoustic Tomography: Analysis and Numerics''. For open access purposes, the author has applied a CC BY public copyright license to any author-accepted manuscript version arising from this submission.

\section*{References}
\renewcommand{\i}{\ii}
\printbibliography[heading=none]

\end{document}